\documentclass[12pt]{amsart}
\usepackage[pdftex]{graphicx}
\usepackage{tikz}
\usepackage{amsmath,amssymb,stmaryrd}
\usepackage{amsthm, amsfonts, bm}
\usepackage{enumerate}
\usepackage{comment}
\usepackage[all]{xy}
\usepackage{color}
\usepackage{hyperref}
\usepackage[margin=27truemm]{geometry}
\SetSymbolFont{stmry}{bold}{U}{stmry}{m}{n}

\usetikzlibrary{positioning}

\usepackage[pagewise]{lineno}

\theoremstyle{plain}
\newtheorem{thm}{Theorem}[section]
\theoremstyle{definition}
\newtheorem{dfn}[thm]{Definition}
\newtheorem{rem}[thm]{Remark}
\newtheorem{ques}[thm]{Question}

\theoremstyle{plain}
\newtheorem{lem}[thm]{Lemma}
\newtheorem{prop}[thm]{Proposition}
\newtheorem{cor}[thm]{Corollary}

\theoremstyle{definition}

\newcommand{\R}{\mathbb{R}}

\newcommand{\Z}{\mathbb{Z}}

\newcommand{\N}{\mathbb{N}}
\newcommand{\G}{\mathbb{G}}

\numberwithin{equation}{section}

\title{Combinatorial knot Floer homology and tangle decompositions}

\author{Hajime Kubota}
\address{Department of Mathematical Sciences Based on Modeling and Analysis, School of Interdisciplinary Mathematical Sciences, Meiji University, 4-21-1 Nakano, Nakano-ku, Tokyo 164-8525 Japan}
\email{kubota\_hajime@meiji.ac.jp}
\date{}

\begin{document}

\begin{abstract}
We compute the top, next-to-top, and top–2 terms of the grid homology for diagonal knots, which are knots represented by grid diagrams with O-markings lying along the diagonal.
This class includes many positive braid knots.
Its next-to-top term detects the number of prime factors, and its top–2 term corresponds to the decompositions of the knot into two non-integer tangles.
We compare diagonal knots with various classes of knots, including positive braid knots, fibered positive knots, and $L$-space knots.
\end{abstract}

\maketitle

\section{Introduction}
Grid homology is a combinatorial reconstruction of knot Floer homology for knots in $S^3$ \cite{oncombinatorial,A-combinatorial-description-of-knot-Floer-homology}.
Sarkar \cite{Grid-diagrams-and-the-Ozsvath-Szabo-tau-invariant} gave a combinatorial proof of the Milnor conjecture \cite{Singular-points-of-complex-hypersurfaces} using grid homology.
Sarkar observed that for positive torus knots, the tau invariant, a concordance invariant in knot Floer homology, can be easily calculated using grid diagrams as in Figure \ref{fig:diagonal}.
Roughly speaking, this is because it is easy to see that the top term of the grid homology of a diagonal knot is one-dimensional.
In this paper, we call the top term the nontrivial homology with the highest Alexander grading.
We also define the next-to-top term, the top-2 term in a similar way.
Arndt-Jansen-McBurney-Vance \cite{Diagonal-knots-and-the-tau-invariant} defined a \textit{diagonal knot} as a knot represented by such a grid diagram.

While grid homology is a purely combinatorial theory and grid homology for any knot is computable in principle, it remains unclear which classes of knots are particularly well-suited to this framework.
In this paper, we compute the top, next-to-top, and top–2 terms of the grid homology for diagonal knots with coefficients in $\mathbb{F}=\Z/2\Z$ (Theorem \ref{thm:str-GH-of-diagonal-knots}).
Furthermore, we argue that the cycles representing the nontrivial homology classes in the next-to-top term correspond to prime factors of the knot, and that the cycles in the top–2 term reflect the decompositions of the knot into two non-integer tangles.
This suggests that diagonal knots are particularly amenable to analysis in the framework of grid homology.

\subsection{Background}
\begin{dfn}[{\cite[Definition 2.2]{Diagonal-knots-and-the-tau-invariant}}]
\label{dfn:diagonal-knot}
A grid diagram is \textit{diagonal} if its $O$-markings are on the diagonal from top left to bottom right, as in Figure \ref{fig:diagonal}.
A knot is called \textit{diagonal} if it is represented by a diagonal grid diagram.
\end{dfn}

\begin{figure}[htbp]
\centering
\includegraphics[scale=0.5]{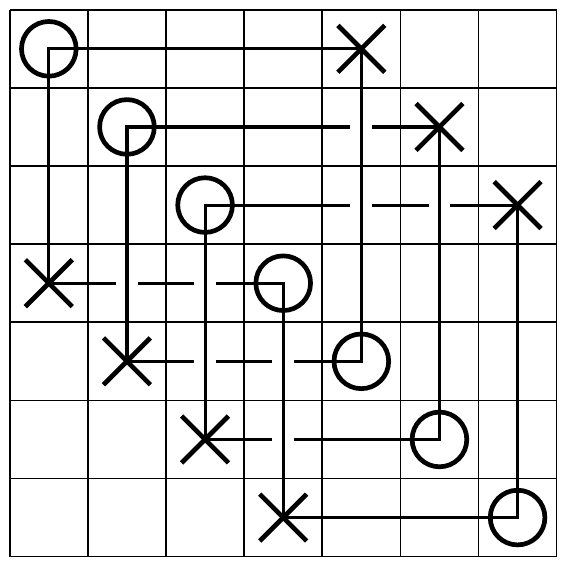}
\caption{A torus knot $T(4,3)$ and a diagonal grid diagram representing it.}
\label{fig:diagonal}
\end{figure}

A knot is called \textit{positive} if it admits a diagram consisting only of positive crossings.

\begin{thm}[{\cite[Theorem 3.3]{Diagonal-knots-and-the-tau-invariant}}]
\label{thm:diagonal-positive}
Diagonal knots are positive.
\end{thm}

It is easy to see that positive torus knots and connected sums of diagonal knots are diagonal.
On the other hand, there are diagonal knots that are neither torus knots nor connected sums of torus knots.
Arndt-Jansen-McBurney-Vance found such a knot represented by a diagonal grid diagram of size $11\times11$, which SnapPy identified as SnapPy census knot $m211$ \cite[Theorem 3.5]{Diagonal-knots-and-the-tau-invariant}.
Furthermore, positive braid knots up to $12$ crossings are diagonal, as {\sf KnotInfo} \cite{knotinfo} provides their diagonal grid diagrams.

The following question is suggested by Vance, one of the authors of \cite{Diagonal-knots-and-the-tau-invariant}.
\begin{ques}
\label{ques:diagonal=positive-braids}
Are diagonal knots the same as positive braid knots?
\end{ques}

Razumovsky \cite{Grid-diagrams-of-Lorenz-links,Some-classes-of-fibered-links} also studied links represented by a diagonal grid diagram.
\textit{Lorenz links} are periodic orbits in the flow in $\R^3$ determined by the Lorenz differential equations, well-known differential equations as a chaotic dynamical system \cite{Lorenz-knots}.
It is known that a Lorenz link has a braid representation as
\begin{equation*}
(\sigma_1\cdots\sigma_{p_1-1})^{q_1}(\sigma_1\cdots\sigma_{p_2-1})^{q_2}\cdots(\sigma_1\cdots\sigma_{p_k-1})^{q_k},
\end{equation*}
for some $2\leq p_1<\dots< p_k$ and $q_i>0$ $(i=1,\dots,k)$ \cite{A-new-twist-on-Lorenz-links}.

\begin{thm}[{\cite{Grid-diagrams-of-Lorenz-links}}]
Lorenz links are diagonal.
\end{thm}

\begin{rem}
Razumovsky defined a diagonal link as a link represented by a diagonal grid diagram satisfying a certain condition and proved that diagonal links are equivalent to Lorenz links.
Razumovsky treated a diagonal grid diagram whose $O$-markings are on the diagonal from top right to bottom left, but it is equivalent to our diagram in the following sense:
if a grid diagram $\G$ represents a link $L$, then the horizontal reflection of $\G$ represents the mirror image $m(L)$, and there is a symmetry of grid homology under mirroring \cite[Proposition 7.1.2]{grid-book}.
\end{rem}

\subsection{Main results}
In this paper, we partially determine the grid homology of diagonal knots.
\textit{Grid homology} of a knot $K\subset S^3$ is a finite-dimensional bigraded $\mathbb{F}$-vector space 
\begin{equation*}
\widehat{GH}(K)=\bigoplus_{d,s\in\mathbb{Z}}\widehat{GH}_d(K,s),
\end{equation*}
which is isomorphic to knot Floer homology $\widehat{HFK}(K)$ as a bigraded vector space,
where $d$ denotes the homological grading called the Maslov grading, $s$ is called the Alexander grading, and $\mathbb{F}=\Z/2\Z$.
It is known that grid homology detects the genus of a knot, as
\begin{equation*}
g(K)=\max\{s|\widehat{GH}_*(K,s)\neq0\}.
\end{equation*}
Grid homology is a categorification of the Alexander polynomial \cite{Holomorphic-disks-and-knot-invariants},
\begin{equation}
\label{eq:categorification}
\Delta_K(t)=\sum_{d,s}(-1)^d \cdot \mathrm{dim}_{\mathbb{F}}\widehat{GH}_d(K,s) \cdot t^s.
\end{equation}
See Section \ref{subsec:grid-cpx} for details of grid homology.

\begin{thm}
\label{thm:str-GH-of-diagonal-knots}
If $K$ is a prime diagonal knot other than the unknot, then we have
\begin{equation}
\label{eq:rank-of-top}
\widehat{GH}_*(K,g(K))\cong 
\begin{cases}
\mathbb{F} & *=0\\
0 & *\neq 0,
\end{cases}
\end{equation}
and
\begin{equation}
\label{eq:rank-of-top-1}
\widehat{GH}_*(K,g(K)-1)\cong
\begin{cases}
\mathbb{F} & *=-1\\
0 & *\neq -1.
\end{cases}
\end{equation}

\end{thm}

We can define the diagonal analogue of the arc index.
Recall that the \textit{arc index} $\alpha(K)$ of a knot $K$ is the minimal size of grid diagrams representing $K$.
The additivity of the arc index is known \cite{Embedding-knots-and-links-in-an-open-book.-I.-Basic-properties}, i.e., 
$\alpha(K_1\# K_2)=\alpha(K_1)+\alpha(K_2)-2$.
It is easy to see that the connected sums of diagonal knots are diagonal.
However, the converse is not obvious.
By exploring the next-to-top term of grid chain complexes, we obtain the following.

\begin{dfn}
For a diagonal knot $K$, we define \textit{diagonal arc index} $\alpha_{\mathrm{diag}}(K)$ as the minimal size of a diagonal grid diagram representing $K$.
\end{dfn}

\begin{thm}
\label{thm:connectsum diagonal}
If the connected sum $ K_1\#K_2$ is diagonal, then both $K_1$ and $K_2$ are diagonal.

For two diagonal knots $K_1, K_2$, then we have 
    \begin{equation*}
    \alpha_{\mathrm{diag}}(K_1\#K_2)=\alpha_{\mathrm{diag}}(K_1)+\alpha_{\mathrm{diag}}(K_2)-2.
    \end{equation*}
\end{thm}

By definition we have $\alpha_{\mathrm{diag}}(K)\geq \alpha(K)$.
There is no known diagonal knot satisfying $\alpha_{\mathrm{diag}}(K)>\alpha(K)$.

\begin{thm}
\label{thm:diagonal top-2}
Let $K$ be a prime diagonal knot other than a $(2,q)$ torus knot or the unknot for any odd integer $q$.
Then, we have
\begin{equation}
\label{eq:rank-of-top-2}
\widehat{GH}_*(K,g(K)-2)\cong
\begin{cases}
\mathbb{F}^m & *=-1\\
0 & *\neq -1,
\end{cases}
\end{equation}
for some integer $m\geq0$.
Furthermore, $K$ admits at least $m$ decompositions into two non-integer tangles.
\end{thm}

In Section \ref{sec:top-2}, we observe that as we fix a diagonal grid diagram for a knot, each nontrivial homology class of the top-2 term naturally gives a decomposition of the knot into two non-integer tangles.
We briefly explain the idea using an example.
The knot $13n_{241}$ is represented by a diagonal grid diagram in Figure \ref{fig:13n241}, and its grid homology at $A=g(K)-2$ is 
\begin{equation*}
\widehat{GH}_*(13n_{241},g(K)-2)\cong
\begin{cases}
\mathbb{F}^2 & *=-1\\
0 & *\neq -1.
\end{cases}
\end{equation*}
In general, the grid chain complex is generated by states, which are $n$-tuples of points on the torus associated with the grid diagram.
Each nontrivial homology class of $\widehat{GH}_{-1}(13n_{241},g(K)-2)$ is represented by the state shown in Figure \ref{fig:13n241_2}.
These states have $n-2$ points on the diagonal, and the other two points determine two square domains corresponding to a decomposition of two non-integer $2$-tangles.
At present, it remains unclear what types of non-integer tangles appear in diagonal grid diagrams.

\begin{figure}[htbp]
\centering
\includegraphics[scale=0.5]{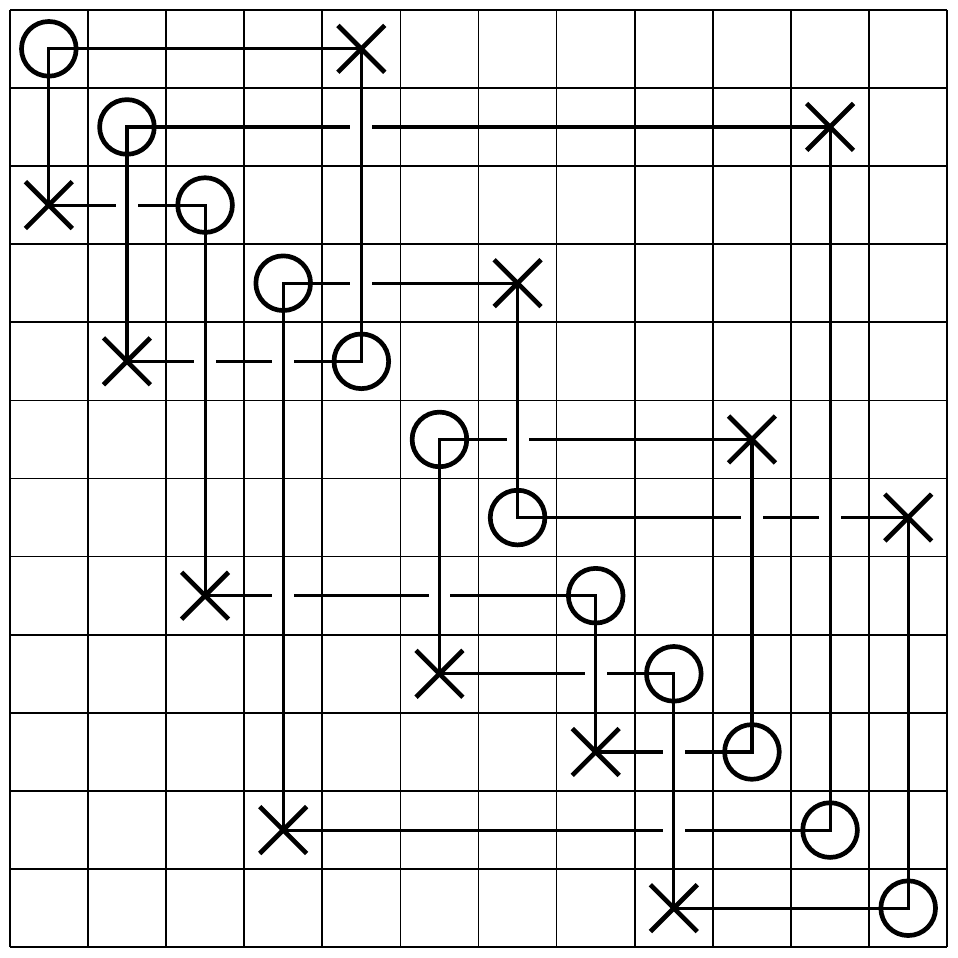}
\caption{A diagonal grid diagram representing $13n_{241}$.}
\label{fig:13n241}
\end{figure}

\begin{figure}[htbp]
\centering
\includegraphics[width=1\linewidth]{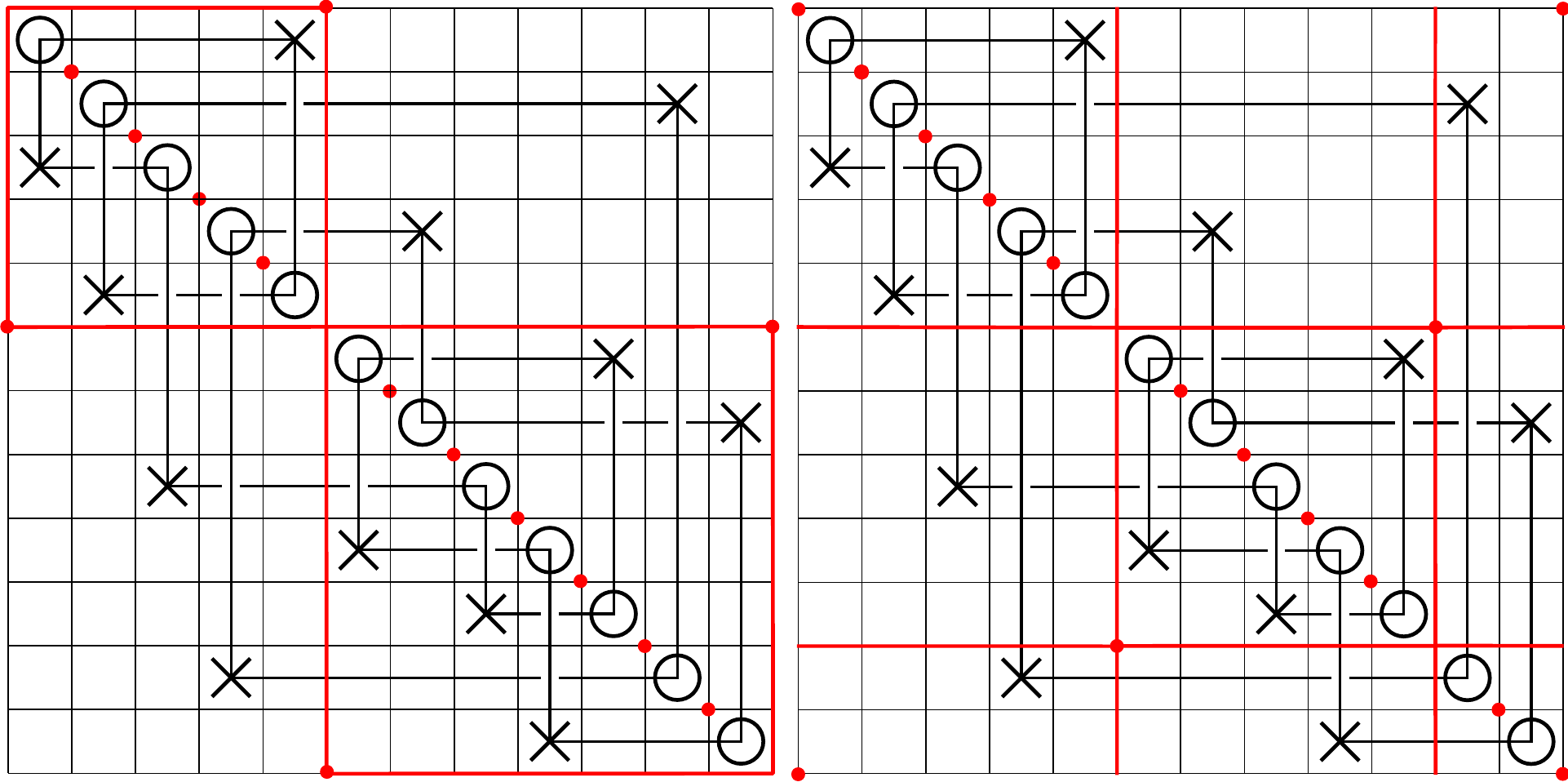}
\caption{Two states and the corresponding square domains representing the nontrivial homology classes of $\widehat{GH}_{-1}(13n_{241},g(K)-2)$.}
\label{fig:13n241_2}
\end{figure}

\begin{cor}
\label{cor:Alexander-of-diagonal}
If $K$ is a prime diagonal knot other than the unknot or $(2,q)$ torus knot, then the symmetrized Alexander polynomial $\Delta_K(t)$ is written as
\begin{equation}
\label{eq:Alexander-poly-of-diagonal}
\Delta_K(K)=t^{g}-t^{g-1} -mt^{g-2}\dots-mg^{-g+2}-t^{-g+1}+t^{-g},
\end{equation}
where $g=g(K)$ is the genus of $K$ and $m$ is some non-negative integer.

\end{cor}
\begin{proof}
This is a direct consequence of Theorem \ref{thm:str-GH-of-diagonal-knots} and Equation \eqref{eq:categorification}.
\end{proof}

\begin{cor}[\cite{Some-classes-of-fibered-links}]
\label{cor:diagonal-fibered}
Diagonal knots are fibered.
\end{cor}
\begin{proof}
This follows from the fact that a knot $K$ is fibered if and only if the top term of $\widehat{HFK}(S^3, K)\cong \widehat{GH}(K)$ is one-dimensional (\cite{Knot-Floer-homology-detects-genus-one-fibred-knots,Knot-Floer-homology-detects-fibred-knots}).
\end{proof}
\begin{rem}
In \cite{Some-classes-of-fibered-links}, Razumovsky proved that if a diagonal grid diagram does not naturally represent a split link, then the corresponding link is fibered.
\end{rem}

\begin{cor}
There are infinitely many prime knots that are positive but not diagonal.
\end{cor}
\begin{proof}
The counterexamples are the pretzel knots $P(p,q,r)$ where $p$, $q$, and $r$ are negative odd integers.
They are positive and have the (symmetrized) Alexander polynomial as
\begin{equation*}
\Delta_{P(p,q,r)}(t)=\frac{1}{4t}\left((pq+qr+rp)(t^2-2t+1)+t^2+2t+1\right).
\end{equation*}
When $p$, $q$, and $r$ are negative odd integers, the coefficient of $t$ is more than one and the genus of $P(p,q,r)$ is one (and thus $P(p,q,r)$ is prime).
Since a fibered knot has a monic Alexander polynomial and its genus is half of the degree of the Alexander polynomial, $P(p,q,r)$ is non-fibered.
By Corollary \ref{cor:diagonal-fibered}, they are not diagonal.
\end{proof}

\begin{prop}
There are fibered positive knots that are not diagonal.
\end{prop}
\begin{proof}
According to {\sf KnotInfo} \cite{knotinfo}, the fibered positive prime knots with up to $12$ crossings whose second coefficient of the Alexander polynomial is not $-1$ are
\begin{equation*}
10_{154}, 10_{161}, 11n_{183}, 12n_{m}\ (m=91,105,136,187,328,417,426,518,591,640,647,694,850).
\end{equation*}
By Corollary \ref{cor:Alexander-of-diagonal}, they are not diagonal.
\end{proof}
\begin{rem}
By {\sf KnotInfo} \cite{knotinfo}, for any fibered positive prime knot, up to $12$ crossings, whose Alexander polynomial is written as \eqref{eq:Alexander-poly-of-diagonal}, Equations \eqref{eq:rank-of-top}, \eqref{eq:rank-of-top-1} and \eqref{eq:rank-of-top-2} hold.
\end{rem}

\begin{ques}
If $K$ is a fibered positive prime knot with Alexander polynomial
\begin{equation*}
\Delta_K(K)=t^{g}-t^{g-1}-mt^{g-2}\dots-mt^{-g+2}-t^{-g+1}+t^{-g},
\end{equation*}
then do Equations \eqref{eq:rank-of-top}, \eqref{eq:rank-of-top-1} and \eqref{eq:rank-of-top-2} hold?
\end{ques}

The next-to-top term of knot Floer homology for fibered knots has been studied in the literature.
Ghiggini-Spano \cite{Knot-Floer-homology-of-fibred-knots-and-Floer-homology-of-surface-diffeomorphisms} and Ni \cite{Knot-Floer-homology-and-fixed-points} studied the relation between the next-to-top term of knot Floer homology and the number of fixed points of the monodromy.
Since prime diagonal knots are fibered and their next-to-top term is one-dimensional, we immediately obtain the following:
\begin{cor}
Any prime diagonal knot admits a monodromy with no fixed point.
\end{cor}
\begin{proof}
Apply Theorem \ref{thm:str-GH-of-diagonal-knots} to \cite[Theorem 1.2]{Knot-Floer-homology-of-fibred-knots-and-Floer-homology-of-surface-diffeomorphisms} or \cite[Theorem 1.2]{Knot-Floer-homology-and-fixed-points}
\end{proof}

\subsection{Positive braids, Lorenz knots, and \texorpdfstring{$L$}{L}-space knots}
We compare diagonal knots to some classes of knots, such as positive braid knots, Lorenz knots, and $L$-space knots.

As described in Question~\ref{ques:diagonal=positive-braids}, diagonal knots are conjectured to be equivalent to positive braid knots.
It is well known that positive braid knots are fibered and positive \cite{Constructions-of-fibred-knots-and-links}, and the same holds for diagonal knots.
Ito \cite[Corollary 1.1]{A-note-on-HOMFLY-polynomial-of-positive-braid-links} determined the top and next-to-top terms of the Alexander polynomial for positive braid knots, and his result is consistent with Corollary~\ref{cor:Alexander-of-diagonal}.
Moreover, Cheng \cite{Knot-Floer-homology-of-positive-braids} showed that Theorem~\ref{thm:str-GH-of-diagonal-knots} holds for any prime positive braid knot.
In other words, the next-to-top term of knot Floer homology for prime positive braid knots is one-dimensional.
This provides supporting evidence for Question~\ref{ques:diagonal=positive-braids}.
We give diagonal grid diagrams for some positive braid knots in Appendix \ref{sec:appendix}.

It is known that $u(K)=g(K)$ if $K$ is a positive braid knot (see \cite{On-unknotting-fibered-positive-knots-and-braids}  for a concise summary of this result).
More generally, it is conjectured that $u(K)=g(K)$ for any fibered positive knot.
Although it is unclear whether positive braid knots are equivalent to diagonal knots, we prove this equation for diagonal knots.

\begin{thm}
\label{thm:diagonal-u=g}
For any diagonal knot $K$, we have $u(K)=g(K)$.
\end{thm}

We remark that the inequality $u(K)\geq g(K)$ holds for all strongly quasi-positive knots \cite{Quasipositivity-as-an-obstruction-to-sliceness}, a class of knots including fibered positive knots.

An \textit{$L$-space} $Y$ is a rational homology three-sphere such that
\begin{equation*}
\dim{\widehat{HF}(Y)}=|H_1(Y;\Z)|,
\end{equation*}
where $\widehat{HF}(Y)$ is Heegaard Floer homology of $Y$.
A knot is called an \textit{$L$-space} if it admits a positive Dehn surgery to an $L$-space.
Torus knots are $L$-space knots.
Theorem \ref{thm:str-GH-of-diagonal-knots} reminds us $L$-space knots.
It is known that $L$-space knots also satisfy \eqref{eq:rank-of-top} and \eqref{eq:rank-of-top-1}.

\begin{thm}[{\cite[Theorem 1.2]{On-knot-Floer-homology-and-lens-space-surgeries},\cite[Theorem 1.1]{The-next-to-top-term-in-knot-Floer-homology}}]
Let $K$ be an $L$-space knot.
Then there exists an integer $d$ such that
    \begin{equation}
\widehat{HFK}_*(S^3, K,g(K))\cong 
\begin{cases}
\mathbb{F} & *=d\\
0 & *\neq d,
\end{cases}
\end{equation}
and
\begin{equation}
\widehat{HFK}_*(S^3,K,g(K)-1)\cong
\begin{cases}
\mathbb{F} & *=d-1\\
0 & *\neq d-1.
\end{cases}
\end{equation}

\end{thm}
Moreover, \cite[Theorem 1.2]{On-knot-Floer-homology-and-lens-space-surgeries} states that, for each $s$, the knot Floer homology of $(S^3,K)$ in Alexander grading $s$ is at most one-dimensional.
On the other hand, Theorem \ref{thm:diagonal top-2} implies that, for prime diagonal knots, the knot Floer homology in Alexander grading $g(K)-2$ can have dimension greater than one.
Therefore, it is easy to see that diagonal knots are not equivalent to $L$-space knots.

\begin{prop}
The knot $13n_{4639}$ is an $L$-space but not diagonal.
\end{prop}
\begin{proof}
The knot $13n_{4639}$ is an $L$-space (by {\sf KnotInfo} \cite{knotinfo}) but not positive (prime positive knots up to $15$ crossings are available in \cite{Knot-data-tables-Stoimenow}).
By Theorem \ref{thm:diagonal-positive}, it is not diagonal.
\end{proof}

\begin{rem}
By {\sf KnotInfo} \cite{knotinfo}, $L$-space knots with up to $12$ crossings are torus knots or Lorenz knots.
$L$-space knots with $13$ crossings are $13a_{4878}$, $13n_{4587}$, and $13n_{4639}$.
Among them, the knot $13a_{4878}=T(2,13)$ is a Lorenz knot, and the others are not.
A diagonal grid diagram representing $13n_{4587}$ is obtained from a grid diagram in {\sf KnotInfo} \cite{knotinfo} by applying row commutations two times (Interchanging the $6$-th row $\leftrightarrow$ the $7$-th row and the $8$-th row $\leftrightarrow$ the $9$-th row).
On the other hand, as described above, $13n_{4639}$ is not diagonal.

There are prime diagonal knots but not $L$-space, see Appendix \ref{sec:appendix}.

\end{rem}

Now the situation is described below.
\begin{equation*}
\xymatrix@C=20pt{
\{\text{Lorenz knots}\} \ar@{=>}[r] \ar@{=>}[dr]  & \{\text{positive braids}\} \ar@{=>}[r]  & \{\text{fibered positive}\} \ar@{=>}[r] & \{\text{fibered SQP knots}\}\\
 &  \{\text{diagonal knots}\} \ar@{.}[u]^? \ar@{=>}[ur]\ar@{-}[rr]^{\nsubseteq \text{and}\nsupseteq} &  & \{\text{$L$-space knots}\}\ar@{=>}[u]
}
\end{equation*}
where SQP stands for strongly quasi-positive \cite{Notions-of-positivity-and-the-Ozsvath-Szabo-concordance-invariant} and $X \Rightarrow Y$ means that $X$ is a proper subset of $Y$.

\subsection{Outline of the paper.}
In Section \ref{sec:grid-homology} we review the basics of grid homology.
In Section \ref{sec:grid-diagonal}, we explain why grid homology is effective for diagonal knots.
In Section \ref{sec:top, top-1}, we prove Theorems \ref{thm:str-GH-of-diagonal-knots} and \ref{thm:connectsum diagonal}.
In Section \ref{sec:combinatoial chain complex}, we introduce some chain complexes that will be used in the proof of Theorem \ref{thm:diagonal top-2}.
In Section \ref{sec:top-2}, we determine the top–2 term of grid homology to prove Theorem \ref{thm:diagonal top-2}.
In Section \ref{sec:unknotting-diagonl}, we prove Theorem \ref{thm:diagonal-u=g}.
Finally, in Section \ref{sec:appendix}, we give diagonal grid diagrams for some positive braid knots.

\subsection{Acknowledgement}
I am grateful to my supervisor, Tetsuya Ito, for helpful discussions and corrections.
I am also indebted to Katherine Vance and Zhechi Cheng for fruitful discussions and comments.
Finally, I would like to thank the anonymous referee for valuable comments that greatly improved the quality of this paper.

This work was supported by JST, the establishment of university fellowships towards the creation of science technology innovation, Grant Number JPMJFS2123.

\section{Preliminaries}
\label{sec:grid-homology}
In this section, we introduce grid diagrams and grid chain complexes.
See \cite{grid-book} for details.

\subsection{Grid diagrams}
A \textit{(toroidal) grid diagram} $\mathbb{G}$ is an $n\times n$ grid of squares on the torus, some of which contain an $X$- or $O$-marking such that:
\begin{enumerate}[(i)]
\item There is exactly one $O$-marking and $X$-marking on each row and column.
\item $O$-markings and $X$-markings do not share the same square.
\end{enumerate}
A grid diagram determines an oriented link in $S^3$.
Drawing oriented segments from the $O$-markings to the $X$-markings in each row and from the $X$-markings to the $O$-markings in each column.
Assume that the vertical segments always cross above the horizontal segments.
Then we obtain an oriented link diagram.
We think that every toroidal diagram is oriented naturally.
The horizontal circles and vertical circles that separate the torus into $n\times n$ squares are denoted by $\boldsymbol{\alpha}=\{\alpha_i\}_{i=1}^n$ and $\boldsymbol{\beta}=\{\beta_j\}_{j=1}^n$ respectively.

Any link can be represented by (toroidal) grid diagrams.
Two grid diagrams representing the same links are connected by a finite sequence of the following moves called grid moves \cite{Embedding-knots-and-links-in-an-open-book.-I.-Basic-properties}:
\begin{itemize}
\item \textbf{Commutation} (Figure \ref{fig:comm}) permuting two adjacent columns (resp. rows) satisfying the following condition:
The projections of the two segments connecting two markings in each column (resp. row), into a single vertical circle $\beta_j$ are either disjoint, or one is contained in the interior of the other.
\item \textbf{(De-)stabilization} (Figure \ref{fig:sta}) let $\mathbb{G}$ be an $n\times n$ grid diagram.
Then $\mathbb{G}'$ is called a stabilization of $\mathbb{G}$ if it is an $(n+1)\times(n+1)$ grid diagram obtained by the following procedure:
Choose a marked square in $\mathbb{G}$, and remove the marking in that square, in the other marked square in its row, and in the other marked square in its column.
Split the row and column into two.
There are four ways to put markings in the two new columns and rows in the $(n+1)\times(n+1)$ grid to obtain a grid diagram.
There are four types of stabilizations when interchanging the roles of $X$'s and $O$'s.
The inverse of a stabilization is called a destabilization.
\end{itemize}

\begin{figure}[htbp]
\centering
\includegraphics[scale=0.5]{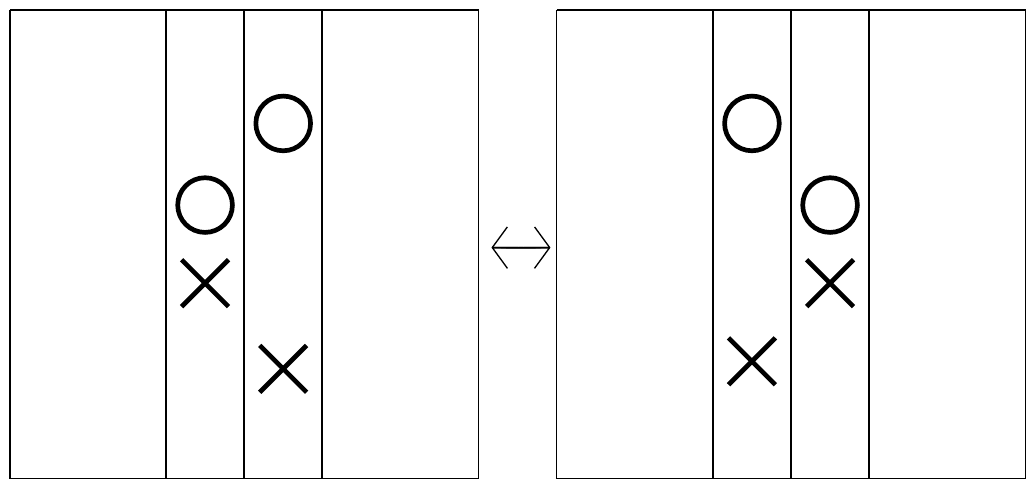}
\caption{Commutation.}
\label{fig:comm}
\end{figure}

\begin{figure}[htbp]
\centering
\includegraphics[scale=0.5]{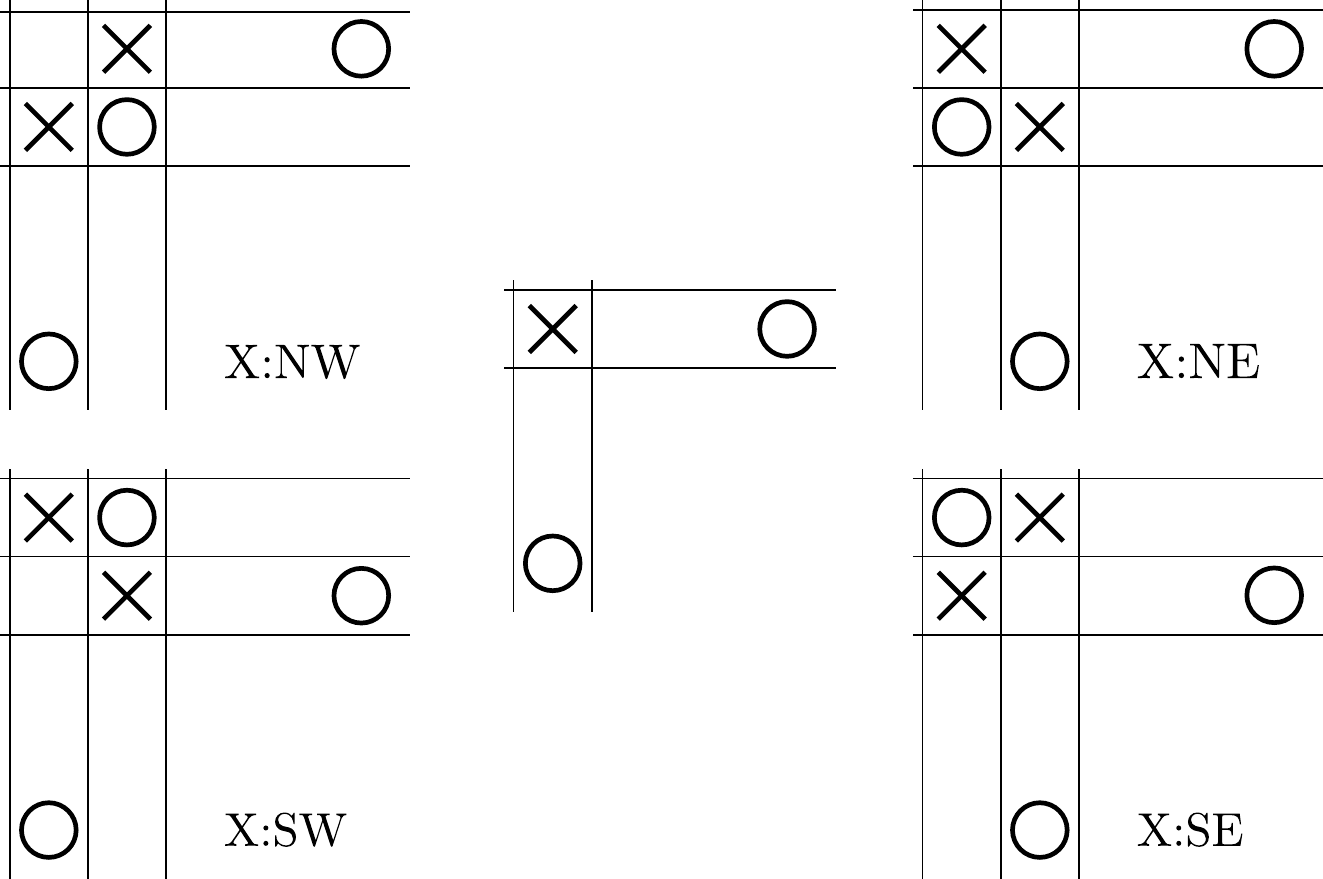}
\caption{There are four types of stabilization at an $X$-marking.}
\label{fig:sta}
\end{figure}

In this paper, we mainly consider grid diagrams representing a knot.

\subsection{Grid chain complexes}
\label{subsec:grid-cpx}
Let $\G$ be an $n\times n$ grid diagram.
A \textit{state} $\mathbf{x}$ of $\mathbb{G}$ is an $n$-tuple of points on the torus such that each horizontal circle $\alpha_i$ has exactly one point and each vertical circle $\beta_j$ has exactly one point of $\mathbf{x}$.
We denote by $\mathbf{S}(\mathbb{G})$ the set of states of $\mathbb{G}$.
For $\mathbf{x,y}\in \mathbf{S}(\mathbb{G})$, a \textbf{domain} $p$ from $\mathbf{x}$ to $\mathbf{y}$ is a formal linear combination of the closures of the squares in $\mathbb{G}\setminus(\boldsymbol{\alpha}\cup\boldsymbol{\beta})$ such that $\partial(\partial_\alpha p)=\mathbf{y}-\mathbf{x}$ and $\partial(\partial_\beta p)=\mathbf{x}-\mathbf{y}$, where $\partial_\alpha p$ is the portion of the boundary of $p$ in the horizontal circles $\alpha_1\cup\dots\cup\alpha_n$ and $\partial_\beta p$ is the portion of the boundary of $p$ in the vertical ones.
A domain $p$ is \textit{positive} if the coefficient of any square is nonnegative.
In this paper, we always consider positive domains such that there is at least one square whose coefficient is zero for each row and column.
We often call such a positive domain non-periodic.
Let $\pi^+(\mathbf{x,y})$ denote the set of such domains from $\mathbf{x}$ to $\mathbf{y}$.
For two states $\mathbf{x,y\in S}(\mathbb{G})$ with $|\mathbf{x\cap y}|=n-2$, an \textit{rectangle} $r\in\pi^+(\mathbf{x},\mathbf{y})$ is a domain such that $\partial r$ is the union of four segments.
Let $\mathrm{Rect}(\mathbf{x,y})$ be the set of rectangles from $\mathbf{x}$ to $\mathbf{y}$.
A rectangle $r\in\mathrm{Rect}(\mathbf{x,y})$ is \textit{empty} if $\mathbf{x}\cap\mathrm{Int}(r)=\mathbf{y}\cap\mathrm{Int}(r)=\emptyset$.
Let $\mathrm{Rect}^\circ(\mathbf{x,y})$ be the set of empty rectangles from $\mathbf{x}$ to $\mathbf{y}$.
If $|\mathbf{x\cap y}|\neq n-2$, then we define $\mathrm{Rect}^\circ(\mathbf{x,y})=\emptyset$.
For two domains $p_1\in\pi(\mathbf{x,y})$ and $p_2\in\pi(\mathbf{y,z})$, the \textit{composite domain} $p_1*p_2$ is the domain from $\mathbf{x}$ to $\mathbf{z}$ such that the coefficient of each square is the sum of the coefficient of the square of $p_1$ and $p_2$.
In this paper, a rectangle is called a square if its width and height are the same.

We denote the set of $O$-markings by $\mathbb{O}$ and the set of $X$-markings by $\mathbb{X}$.
We will use the labeling of markings as $\{O_i\}_{i=1}^n$ and $\{X_j\}_{j=1}^n$.

A \textit{planar realization} of a toroidal diagram $\G$ is a planar figure obtained by cutting it along some $\alpha_i$ and $\beta_j$ and putting it on $[0,n)\times[0,n)\subset\mathbb{R}^2$ in a natural way.

For two points $(a_1,a_2),(b_1,b_2)\in\mathbb{R}^2$, we say $(a_1,a_2)<(b_1,b_2)$ if $a_1<b_1$ and $a_2<b_2$.
For two sets of finitely points $A,B\subset\mathbb{R}^2$, let $\mathcal{I}(A,B)$ be the number of pairs $a\in A,b\in B$ with $a<b$ and let $\mathcal{J}(A,B)=(\mathcal{I}(A,B)+\mathcal{I}(B,A))/2$.
We consider that $n$ points of states are on lattice point on $\mathbb{R}^2$ and each $O$- and $X$-marking is located at $(l+\frac{1}{2},m+\frac{1}{2})$ for some $l,m\in\{0,1,\dots,n-1\}$.
We regard a state $\mathbf{x}\in\mathbf{S}(\G)$ as a set of $n$ lattice points in $[0,n)\times[0,n)\subset\mathbb{R}^2$.

\begin{dfn}
\label{def-maslov-alexander-link}
Take a planar realization of $\mathbb{G}$.
For $\mathbf{x\in S}(\mathbb{G})$, the Maslov grading $M(\mathbf{x})$ and the Alexander grading $A(\mathbf{x})$ are defined by
\begin{align}
M(\mathbf{x})&=\mathcal{J}(\mathbf{x}-\mathbb{O},\mathbf{x}-\mathbb{O})+1,
\label{eq:Maslov-link}
\\
A(\mathbf{x})&=\mathcal{J}(\mathbf{x},\mathbb{X}-\mathbb{O})+\frac{1}{2}\mathcal{J}(\mathbb{O}+\mathbb{X},\mathbb{O}-\mathbb{X})-\frac{n-1}{2}.
\label{eq:Alexander-link}
\end{align}
\end{dfn}
Both the Maslov grading and the Alexander grading are well-defined as a toroidal diagram \cite[Lemma 2.4]{oncombinatorial}.

For a positive domain $D\in\pi^+(\mathbf{x},\mathbf{y})$ and $i=1,\dots,n$, let $O_i(D)$ be the coefficient of the square containing $O_i$ and let
\begin{equation*}
|D\cap\mathbb{O}|=\sum_{i=1}^nO_i(D).
\end{equation*}
We define $|D\cap\mathbb{X}|$ in the same way.
The Maslov and Alexander gradings satisfy 
\begin{equation}
\label{eq:Maslov-difference}
M(\mathbf{x})-M(\mathbf{y})=1-2|r\cap\mathbb{O}|+2|\mathbf{x}\cap\mathrm{Int}(r)|,
\end{equation}
and
\begin{equation}
\label{eq:Alexander-difference}
A(\mathbf{x})-A(\mathbf{y})=|r\cap\mathbb{X}|-|r\cap\mathbb{O}|.
\end{equation}
for any rectangle $r\in\mathrm{Rect}(\mathbf{x},\mathbf{y})$. \cite[Proposition 4.3.3]{grid-book}

In general, for a positive, non-periodic domain $D\in\pi^+(\mathbf{x},\mathbf{y})$, there is a sequence of states
$\mathbf{x}=\mathbf{x}_0,\mathbf{x}_1,\dots,\mathbf{x}_k=\mathbf{y}$ and rectangles $r_i\in\mathrm{Rect}(\mathbf{x}_{i-1},\mathbf{x}_i)$ such that $r_1*\dots*r_k=D$.
Then, we have
\begin{equation}
\label{eq:Alexander-difference2}
A(\mathbf{x})-A(\mathbf{y})=\sum(|r_i\cap\mathbb{X}|-|r_i\cap\mathbb{O}|)=|D\cap\mathbb{X}|-|D\cap\mathbb{O}|.
\end{equation}

In the original construction, the minus version of grid homology is defined first, and the hat and tilde versions are then obtained from it \cite{grid-book}.
Since only the hat version of grid homology is used in the main results, we define the hat and tilde versions without introducing the minus version.
The hat version $\widehat{GH}(K)$ is an invariant of knots, but the tilde version $\widetilde{GH}(K)$ is not.
However, $\widehat{GH}(K)$ is recovered from $\widetilde{GH}(K)$ by applying a correction depending on the grid size.
Sometimes it is convenient to deal with the tilde version because the tilde grid complex is much simpler than the hat grid complex.
The precise definitions of the minus, hat, and tilde versions are described in \cite{grid-book}.

\begin{dfn}
The grid chain complex $\widetilde{GC}(\G)$ is an $\mathbb{F}$-vector space generated by $\mathbf{S}(\G)$, with the endmorphism defined by
\begin{equation*}
\widetilde{\partial}(\mathbf{x})=\sum_{\mathbf{y}\in\mathbf{S}(g)}\#\{r\in\mathrm{Rect}^\circ(\mathbf{x,y})|r\cap\mathbb{O}=r\cap\mathbb{X}=\emptyset\}\cdot\mathbf{y},
\end{equation*}
where $\#\{\cdot\}$ counts rectangles modulo $2$.
The grid chain complex $\widetilde{GC}(\G)$ is a bigraded chain complex, i.e., $\widetilde{GC}(\G)$ is a bigraded vector space with respect to the Maslov and Alexander grading, and the differential $\widetilde{\partial}$ satisfies $\widetilde{\partial}\circ\widetilde{\partial}=0$ and drops the Maslov grading by one, and preserves the Alexander grading.

The homology of $(\widetilde{GC}(\G),\widetilde{\partial})$ is denoted by $\widetilde{GH}(\G)$.
\end{dfn}

Let $W$ be a two-dimensional graded vector space $W\cong\mathbb{F}_{0,0}\oplus\mathbb{F}_{-1,-1}$.
For a bigraded $\mathbb{F}$-vector space $X$, the \textbf{shift} of $X$, denoted $X\llbracket a,b\rrbracket$, is the bigraded $\mathbb{F}$-vector space so that $X\llbracket a,b\rrbracket_{d,s}=X_{d+a,s+b}$.
Then, we have 
\begin{equation*}
X\otimes W\cong X\oplus X\llbracket1,1\rrbracket.
\end{equation*}

\begin{thm}[{\cite[Theorem 1.2, Proposition 2.15]{oncombinatorial}}]
\label{thm:hat-tilde}
For an $n\times n$ grid diagram representing a knot $K$, let $\widehat{GH}(\G)$ be the bigraded vector space defined by
\begin{equation*}
\widetilde{GH}(\mathbb{G})\cong \widehat{GH}(\mathbb{G})\otimes W^{\otimes(n-1)}.
\end{equation*}
The bigraded vector space $\widehat{GH}(\mathbb{G})$ is a knot invariant and often denoted by $\widehat{GH}(K)$.
\end{thm}

\begin{rem}[{\cite[Remark 4.6.13]{grid-book}}]
The hat version of the grid chain complex $\widehat{GC}(\G)$ for a knot can be written explicitly.
For an $n \times n$ grid diagram $\G$, $\widehat{GC}(\G)$ is an $\mathbb{F}$-vector space with basis $\{V_{1}^{k}\cdots V_{n-1}^{k_{n-1}}\cdot\mathbf{x}|k_i\geq0,\mathbf{x\in S}(g)\}$ with the endmorphism defined by
\begin{equation*}
\widehat{\partial}(\mathbf{x})=\sum_{\mathbf{y}\in\mathbf{S}(g)}\left(
\sum_{\{r\in \mathrm{Rect}^\circ(\mathbf{x,y})|r\cap\mathbb{X}=r\cap\{O_n\}=\emptyset\}}
V_{1}^{O_{1}(r)}\cdots V_{n-1}^{O_{n-1}(r)}
\right)\mathbf{y},
\end{equation*}
where $O_i(r)=1$ if $r$ contains $O_i$ and $O_i(r)=0$ otherwise.
For $i=1,\dots,n-1$, we set
\begin{equation*}
M(V_i)=-2, A(V_i)=-1.
\end{equation*}
The homology of $(\widehat{GC}(\G),\widehat{\partial})$ is $\widehat{GH}(\G)$.
\end{rem}

\section{Grid homology and diagonal knots}
\label{sec:grid-diagonal}
In this section, we explain why grid homology is particularly effective for diagonal knots.

\begin{dfn}
\label{dfn:x0}
For an $n \times n$ grid diagram $\G$, let $\mathbf{x}_0(\G)$ be the state of $\G$ consisting of northwest corners of the squares containing an $O$-marking (Figure \ref{fig:x0}).
\end{dfn}
We often write $\mathbf{x}_0(\G)$ simply as $\mathbf{x}_0$.

\begin{figure}[htbp]
\centering
\includegraphics[scale=0.5]{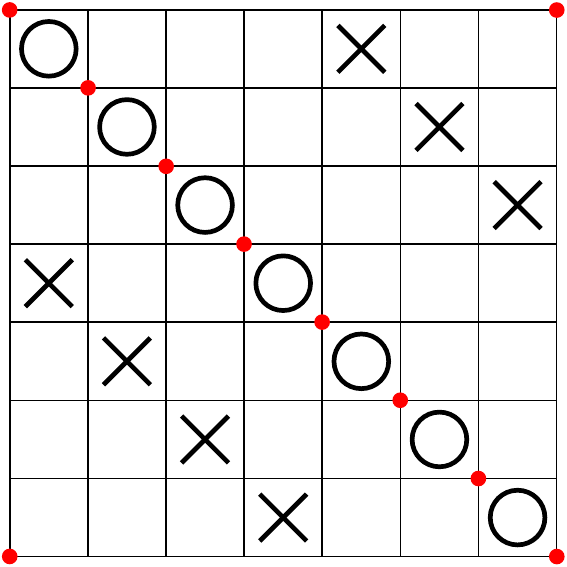}
\caption{The state $\mathbf{x}_0(\G)$.}
\label{fig:x0}
\end{figure}

\begin{lem}
\label{lem:M0}
Let $\G$ be a diagonal grid diagram.
Let $\mathbf{x}$ be any state other than $\mathbf{x}_0$.
\begin{enumerate}
    \item $M(\mathbf{x}_0)=0$.
    \item $M(\mathbf{x})<M(\mathbf{x}_0)$.
    \item There are non-periodic positive domains in $\pi^+(\mathbf{x},\mathbf{x}_0)$.
    \item $M(\mathbf{x})=-1$ is and only if $|(\mathbf{x}\cup\mathbf{x}_0)\setminus(\mathbf{x}\cap\mathbf{x}_0)|=4$ and the two rectangles of $\mathrm{Rect}(\mathbf{x},\mathbf{x}_0)$ are squares.
\end{enumerate}
\end{lem}
\begin{proof}
In general, the equation $M(\mathbf{x}_0)=0$ is verified for all grid diagrams in \cite[Proposition 4.3.1]{grid-book}.

We prove (2) by induction on the size of diagonal grid diagrams, denoted by $l$.
The case $l=2,3$ is straightforward.
Let $\G$ be an $(l+1)\times(l+1)$ diagonal grid diagram and $\G'$ be an $l\times l$ diagonal grid diagram.
We remark that as long as we consider the Maslov grading, we can ignore the $X$-markings.
We write the Maslov grading of $\G$ and $\G'$ as $M_\G$ and $M_{\G'}$ respectively. 
Take any state $\mathbf{x}\in\mathbf{S}(\G)$ other than $\mathbf{x}_0(\G)$.
First, we suppose that $\alpha_{n+1}\cap\beta_2\in\mathbf{x}$.
Then the set of points $\mathbf{x}\setminus\{\alpha_{n+1}\cap\beta_2\}$ naturally determines a state $\mathbf{x}'\in\mathbf{S}(\G)$.
A direct computation shows that $M_\G(\mathbf{x})=M_{\G'}(\mathbf{x}')<M_{\G'}(\mathbf{x}_0(\G'))=0$.
Next, we suppose that $\alpha_{n+1}\cap\beta_2\notin\mathbf{x}$.
If $\alpha_{n+1}\cap\beta_1\in\mathbf{x}$, then there is a state $\mathbf{y}$ and a rectangle $\mathrm{Rect}(\mathbf{y},\mathbf{x})$ such that $\alpha_{n+1}\cap\beta_2\in\mathbf{y}$ and $|r\cap\mathbb{O}|=0$.
Then we have $M_{\G}(\mathbf{x})=M_\G(\mathbf{y})-1$ and the above argument implies that $M_\G(\mathbf{x})<0$.
If $\alpha_{n+1}\cap\beta_1\notin\mathbf{x}$, then there is a state $\mathbf{z}$ and a rectangle $\mathrm{Rect}(\mathbf{x},\mathbf{z})$ such that $\alpha_{n+1}\cap\beta_2\in\mathbf{z}$ and $|r\cap\mathbb{O}|>|\mathbf{x}\cap\mathrm
Int(r)|$.
Then we have $M_{\G}(\mathbf{x})<M_\G(\mathbf{z})$ and the above argument implies that $M_\G(\mathbf{x})<0$.
The same argument implies (4).

The statement (3) is obvious.
We remark that there is such a domain whose coefficients in the squares of the top row and the rightmost column are zero.
\end{proof}

In general, the condition that $\mathbf{x}_0$ attains the maximal Maslov grading does not hold.
For example, for the $4\times4$ grid diagram $\G$ representing the unknot as in Figure \ref{fig:unknot4}, the state $\mathbf{x}=\{\alpha_2\cap\beta_1, \alpha_3\cap\beta_2, \alpha_4\cap\beta_3, \alpha_1\cap\beta_4\}$ satisfies $M(\mathbf{x})=1>0=M(\mathbf{x}_0)$.

\begin{figure}
    \centering
    \includegraphics[scale=0.5]{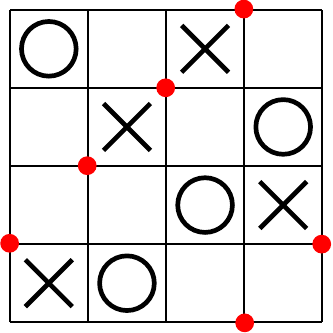}
    \caption{The grid diagram representing the unknot and the state $\mathbf{x}$ with $M(\mathbf{x})=1>0=M(\mathbf{x}_0)$.}
    \label{fig:unknot4}
\end{figure}

\begin{dfn}
\label{dfn:associated domain}
For a state $\mathbf{x}\neq\mathbf{x}_0$, a non-periodic positive domain $D\in\pi^+(\mathbf{x},\mathbf{x}_0)$ is called an \textit{associated domain} of $\mathbf{x}$.
\end{dfn}
There are usually multiple associated domains for $\mathbf{x}$.

\begin{lem}
\label{lem:Amax}
If a diagonal grid diagram $\G$ represents a knot, then we have $A(\mathbf{x})<A(\mathbf{x}_0)$ for $\mathbf{x}\neq \mathbf{x}_0$.
\end{lem}

\begin{proof}
For a state $\mathbf{x}$, there is an associated domain $D\in\pi^+(\mathbf{x},\mathbf{x}_0)$ such that $D$ does not overlap the horizontal circle $\alpha_1$ and the vertical circle $\beta_1$.
For every point $x_i\in\mathbf{x}\setminus\mathbf{x}_0$, draw horizontal and vertical segments from $x_i$ to points of $\mathbf{x}_0$ that do not intersect with $\alpha_1$ or $\beta_1$.
Then the area enclosed by these segments determines the desired domain.
Let $D'$ be a connected component of $D$.
since the points of $\mathbf{x}_0$ are on the diagonal, we have $|D'\cap\mathbb{O}|\geq|D'\cap\mathbb{X}|$.
Since $\G$ represents a knot, $|D'\cap\mathbb{O}|=|D'\cap\mathbb{X}|$ implies that $D'$ is trivial or covers the entire $\G$.
Then Equation \eqref{eq:Alexander-difference} completes the proof.
\end{proof}

\begin{lem}
\label{lem:diagonal-differantial}
Suppose that a diagonal grid diagram $\G$ does not represent a split link.
For any two points $x_i,x_j\in\mathbf{x}_0$, any state $\mathbf{x}$, and an empty rectangle $r\in\mathrm{Rect}^\circ(\mathbf{x}_0,\mathbf{x})$ two of whose corners are $x_i,x_j$, we have $r\cap\mathbb{X}\neq\emptyset$.
\end{lem}
\begin{proof}
If $r$ contains no $X$-marking, then the square domains in $\mathrm{Rect}^\circ(\mathbf{x},\mathbf{x}_0)$ naturally represents a split link component.
\end{proof}
This Lemma implies that, for any state $\mathbf{x}$ and any empty rectangle $r$ counted by the differential $\widetilde{\partial}(\mathbf{x})$, at most one of the corner points $\mathbf{x}\cap\partial r$, the northeast or southwest corner, is on the diagonal.
This makes it easier to examine the chain complex.

\section{The top and text-to-top terms of grid homology}
\label{sec:top, top-1}
We will prove Theorems \ref{thm:str-GH-of-diagonal-knots} and \ref{thm:connectsum diagonal} using Theorem \ref{thm:hat-tilde}.

\begin{proof}[Proof of Equation \eqref{eq:rank-of-top} in Theorem \ref{thm:str-GH-of-diagonal-knots}]
By Lemma \ref{lem:M0} and Proposition \ref{lem:Amax}, the grid complex at $A=g(K)$ is
\begin{equation*}
\cdots\rightarrow 0\rightarrow\mathbb{F}\{\mathbf{x}_0\}\rightarrow 0 \rightarrow \cdots.\qedhere
\end{equation*}
\end{proof}

\begin{proof}[Proof of Equation \eqref{eq:rank-of-top-1} in Theorem \ref{thm:str-GH-of-diagonal-knots}]
Let $\G$ be an $n\times n$ diagonal grid diagram representing a prime diagonal knot $K$ other than the unknot.
By applying grid moves to $\G$, we can assume that 
\begin{enumerate}[(A)]
    \item Every $X$-marking is not adjacent to the $O$-markings, and
    \label{assumption:A}
    \item $n$ is the smallest among diagonal grid diagrams representing $K$.
    \label{assumption:B}
\end{enumerate}
We recall that the state $\mathbf{x}_0=\mathbf{x}_0(\G)$ (Definition \ref{dfn:x0}) satisfies $A(\mathbf{x})<A(\mathbf{x}_0)$ if $\mathbf{x}\neq\mathbf{x}_0$.
Let $\mathbf{x}_1,\dots,\mathbf{x}_n$ be the states obtained from $\mathbf{x}_0$ by switching the height of adjacent two points.

We will see
\begin{equation*}
(M(\mathbf{x}_i),A(\mathbf{x}_i))=(M(\mathbf{x}_0)-1,A(\mathbf{x}_0)-1)=(-1,A(\mathbf{x}_0)-1)
\end{equation*}
for $i=1,\dots,n$, and $A(\mathbf{x})\leq A(\mathbf{x}_0)-2$ if $\mathbf{x}\neq \mathbf{x}_0, \mathbf{x}_1,\dots,\mathbf{x}_n$.
If this is verified, then the grid homology $\widetilde{GH}(\G)$ with the Alexander grading $A(\mathbf{x}_0)-1=g(K)-1$ is
\begin{equation*}
\widetilde{GH}_*(\G, g(K)-1)\cong 
\begin{cases}
\mathbb{F}^n & *=-1\\
0 & *\neq -1,
\end{cases}
\end{equation*}
and Theorem \ref{thm:hat-tilde} completes the proof.

Let $\mathbf{x}$ be a state with $\mathbf{x}\neq\mathbf{x}_0$.
Any associated domain $D\in \pi^+(\mathbf{x}, \mathbf{x}_0)$ of $\mathbf{x}$ satisfies \eqref{eq:Alexander-difference}.
It is sufficient to prove that if $A(\mathbf{x})=A(\mathbf{x}_0)-1$, the domain $D$ must be a single $1\times1$ or $(n-1)\times(n-1)$ square.
Let $m=\#(\mathbf{x}_0\cap D)$ be the number of $90^\circ$ or $270^\circ$ corners of $D$ coming from $\mathbf{x}_0$.
A $90^\circ$ corner of $\mathbf{x}_0\cap D$ is the northwest or southeast corner of $D$.

First, we consider the case when $m=2$, in other words, the domain $D$ is a square (possibly larger than a $1\times1$ square).
We will see that for any square $D\in\pi^+(\mathbf{x}, \mathbf{x}_0)$ for some state $\mathbf{x}$, we have
\begin{equation}
\label{eq:multiplication}
|D\cap\mathbb{O}|-|D\cap\mathbb{X}|=
\begin{cases}
1 & \text{if $D$ is  a $1\times1$ or $(n-1)\times(n-1)$ square}\\
>1 & \text{otherwise}.
\end{cases}
\end{equation}
Since there are exactly two squares $D$ and $D'$ from $\mathbf{x}$ to $\mathbf{x}_0$ (both satisfy Equation \eqref{eq:Alexander-difference}), we can assume that $D$ is the smaller one.
Let $D$ be a $l\times l$ square.
\begin{enumerate}
    \item If $l=1$, obviously we have $|D\cap\mathbb{O}|-|D\cap\mathbb{X}|=1$.
    \item If $l=2$, by Assumption (\ref{assumption:A}), the square $D$ must contain two $O$-markings but no $X$-marking.
    Therefore we have $|D\cap\mathbb{O}|-|D\cap\mathbb{X}|=2$.
    \item
    \label{case:l=3}
    If $l=3$, suppose that $|D\cap\mathbb{O}|-|D\cap\mathbb{X}|=1$.
    By Assumption (\ref{assumption:A}), $\G$ represents a link with more than one component (Figure \ref{fig:D-3x3}).
    This contradicts our assumption that $\G$ represents a knot.
    \item
    \label{case:l>3}
    If $l>3$, consider a new $(l+1)\times(l+1)$ diagonal grid diagram $\G(D)$ obtained from $D$ by adding one row and column and putting one $O$-marking and two $X$-markings  (Figure \ref{fig:G(D)}).
    By Assumption (\ref{assumption:B}), the diagram $\G(D)$ does not represent the unknot.
    By the same argument, the $(n-l+1)\times(n-l+1)$ grid diagram $\G(D')$ obtained from $D'$ in the same way also does not represent the unknot.
    This contradicts our assumption that $\G$ represents a prime knot.
\end{enumerate}

\begin{figure}[htbp]
\centering
\includegraphics[scale=0.5]{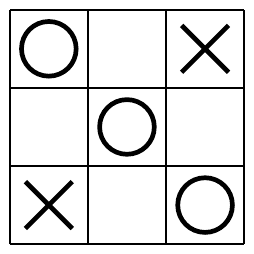}
\caption{A $3\times3$ square $D$ appearing in the case \eqref{case:l=3}.}
\label{fig:D-3x3}
\end{figure}

\begin{figure}[tb]
\centering
\begin{tikzpicture}

\node[anchor=south west,inner sep=0] (img)
  {\includegraphics[scale=0.5]{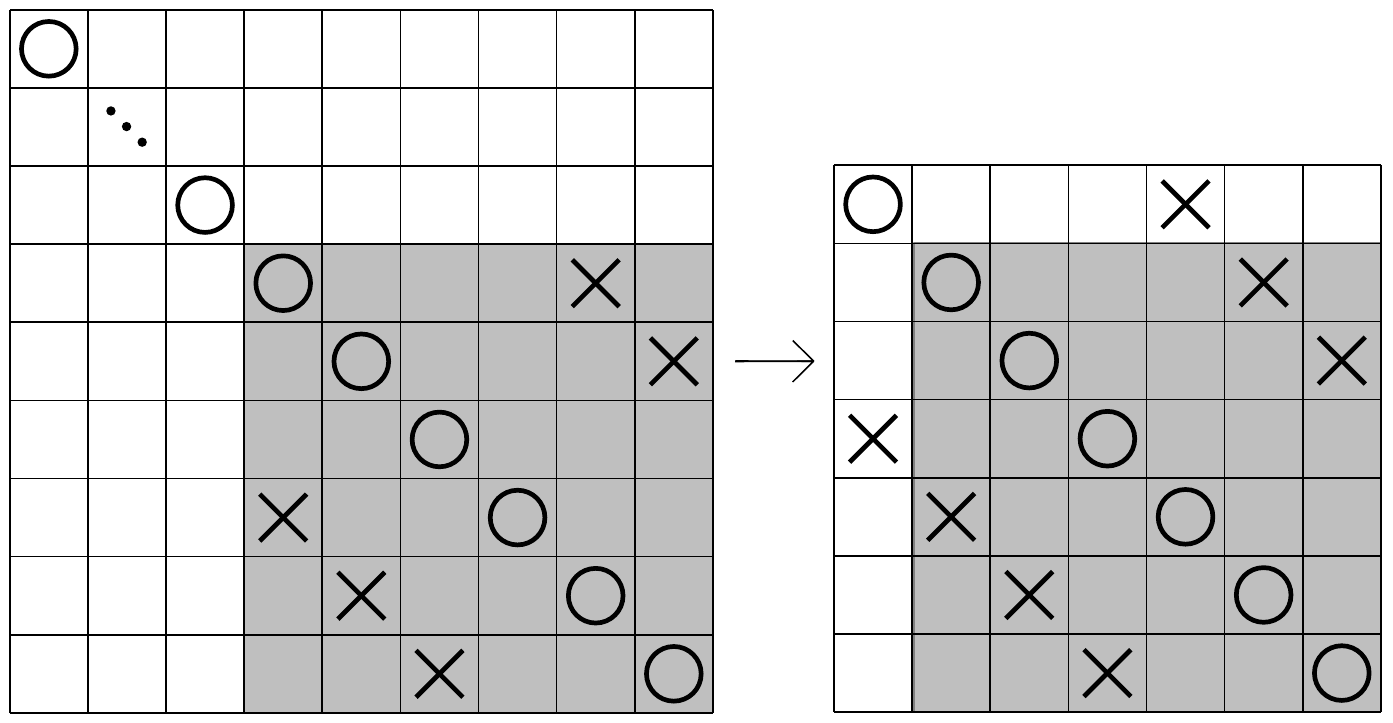}};

\begin{scope}[x={(img.south east)},y={(img.north west)}]

\node at (.255,-.05) {\Large$\G$};
\node at (.80,-.05) {\large$\G(D)$};

\end{scope}

\end{tikzpicture}

\caption{A square $D$ (gray area) in $\G$ and a resulting diagram $\G(D)$ in the case \eqref{case:l>3}.}
\label{fig:G(D)}
\label{fig:example}

\end{figure}

Next, we will consider the case $m>2$.
For any domain $D\in\pi^+(\mathbf{x},\mathbf{x}_0)$, there is a positive domain $E\in\pi^+(\mathbf{y},\mathbf{x})$ such that the composite domain $D'=E*D$ is the sum of square domains two of whose corners are the points of $\mathbf{x}_0$ (Figure \ref{fig:D-D'}).
Since $m>2$, $D'$ consists of more than one square or is a square larger than $1\times1$ and smaller than $(n-1)\times(n-1)$.
Since the diagonal of each square consisting of $D'$ is on the diagonal of $\G$, we can apply Equation \eqref{eq:multiplication} to $D'$.
Then we have
\begin{equation*}
|D'\cap\mathbb{O}|-|D'\cap\mathbb{X}|>1.
\end{equation*}
By the construction of $D'$, we have $|D\cap\mathbb{O}|=|D'\cap\mathbb{O}|$ and $|D\cap\mathbb{X}|\leq|D'\cap\mathbb{X}|$.
Therefore we have
\begin{equation*}
|D\cap\mathbb{O}|-|D\cap\mathbb{X}|\geq|D'\cap\mathbb{O}|-|D'\cap\mathbb{X}|>1.\qedhere
\end{equation*}
\end{proof}

\begin{figure}[htbp]
\centering
\includegraphics[width=1\linewidth]{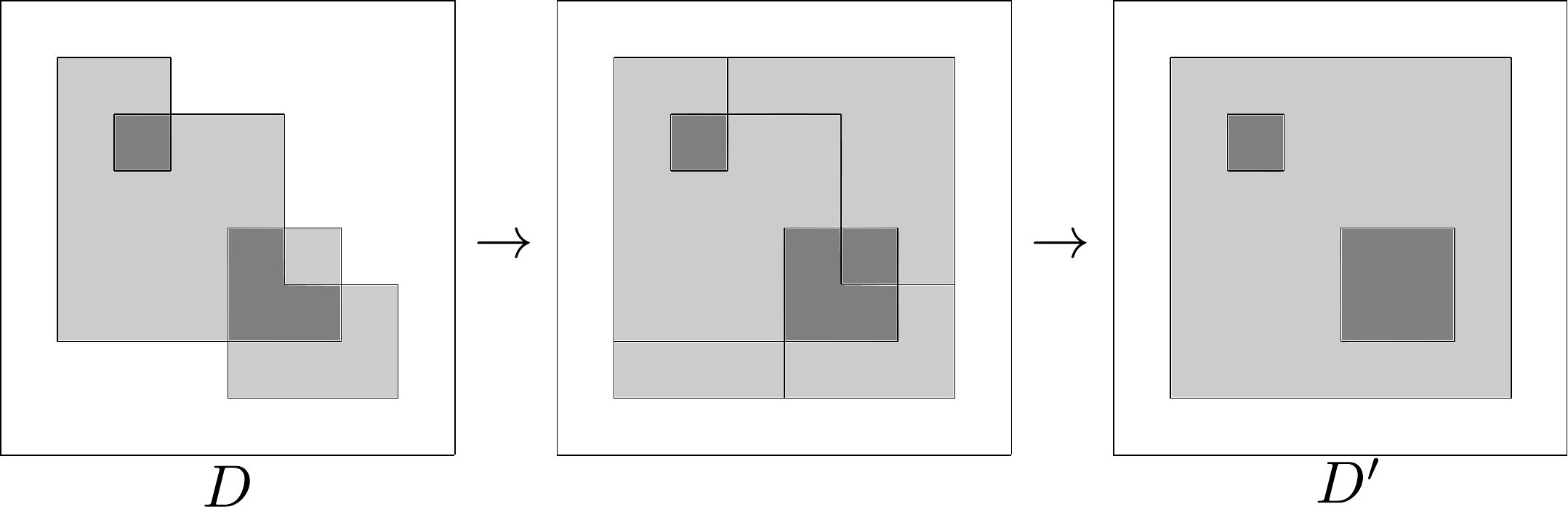}
\caption{An example of domains $D$ and $D'$. The coefficients of squares are one in light gray regions and two in gray regions.}
\label{fig:D-D'}
\end{figure}

\begin{proof}[Proof of Theorem \ref{thm:connectsum diagonal}]
Suppose that $K=K_1\#\cdots\#K_s$ is represented by an $n\times n$ diagonal grid diagram $\G$, where $s>1$ and $K_1,\dots,K_s$ are prime.
Again, we assume the conditions (\ref{assumption:A}) and (\ref{assumption:B})
There is a K\"{u}nneth formula of connected sum for knot Floer homology \cite{Holomorphic-disks-and-knot-invariants}.
This is also verified in the framework of grid homology \cite{Grid-homology-for-spatial-graphs-and-a-Kunneth-formula-of-connected-sum}.
This formula implies that if $K$ is the connected sum of $s$ prime diagonal knots, then we have
\begin{equation*}
\widehat{GH}_*(K,g(K)-1)\cong
\begin{cases}
\mathbb{F}^s & *=-1\\
0 & *\neq -1
\end{cases}
\end{equation*}
and therefore
\begin{equation*}
\widetilde{GH}_*(K,g(K)-1)\cong
\begin{cases}
\mathbb{F}^{s+n-1} & *=-1\\
0 & *\neq -1,
\end{cases}
\end{equation*}
by Theorem \ref{thm:hat-tilde}.

Recall the $n$ states $\mathbf{x}_1,\dots,\mathbf{x}_n$ introduced in the first half of Sections \ref{sec:top, top-1}.
They satisfy $M(\mathbf{x}_i)=-1$ and $A(\mathbf{x}_i)=g(K)-1$, and there is no state $\mathbf{x}$ with $M(\mathbf{x})>-1$ and $A(\mathbf{x})=g(K)-1$.
Obviously we have that $\widetilde{\partial}(\mathbf{x}_i)=0$ for each $i$.
Then there are at least $s-1$ states in $\widetilde{GC}_{-1}(K,g(K)-1)$ other than $\mathbf{x}_1,\dots,\mathbf{x}_n$.
Let $\mathbf{y}$ be such a state.
By Lemma \ref{lem:M0}, there are exactly two associated domains $D,D'\in\mathrm{Rect}(\mathbf{y},\mathbf{x}_0)$ that are squares.
Both $D$ and $D'$ are larger than a $1\times1$ square.
Among such states and square domains, let $\mathbb{D}$ be one of the smallest domains.
By Assumption \eqref{assumption:B}, the diagonal grid diagram $\G(\mathbb{D})$ does not represent the unknot.
The minimality of $\mathbb{D}$ implies that the square domain $D$ of $\G(\mathbb{D})$ with $|D\cap\mathbb{O}|=|D\cap\mathbb{X}|+1$ must be a $1\times 1$ square.
By the same argument as the previous proof, the next-to-top term of $\widehat{GH}(\G(\mathbb{D}))$ is one-dimensional.
Therefore, $\G(\mathbb{D})$ represents a prime knot other than the unknot.
Replace $\mathbb{D}$ in $\G$ with a $1\times 1$ square with an $O$-marking.
The obtained diagonal grid diagram is also denoted by $\G$.
Again, we can assume that $\G$ satisfies the two assumptions \eqref{assumption:A} and \eqref{assumption:B}.

By repeating this procedure, we can obtain diagonal grid diagrams representing each of $K_1,\dots,K_s$.
Furthermore, the obtained grid diagrams $\G(\mathbb{D})$ are minimal.
Therefore, the above argument implies $\alpha_{\mathrm{diag}}(K\#K')=\alpha_{\mathrm{diag}}(K)+\alpha_{\mathrm{diag}}(K')-2$.
\end{proof}

\section{Combinatorial chain complexes}
\label{sec:combinatoial chain complex}
This section provides some kinds of acyclic chain complexes that appear as subcomplexes of $\widetilde{GC}(\G)$.

\subsection{The complex of partitions}
Fix $N\in\N$.
Let $\mathrm{Part}(N)$ be the set of ordered partitions of $N$ as sums of positive integers.
An element $\lambda\in\mathrm{Part}(N)$ has the form 
\begin{equation*}
\lambda=(\lambda_1,\dots,\lambda_m),\ m\in\N,\ \lambda_1,\dots,\lambda_m\in\N,\ \sum_i\lambda_i=N.
\end{equation*}
Let $\ell(\lambda)=m$ be the length of the partition and $\mathrm{Part}(N,m)=\{\lambda\in\mathrm{Part}(N)|\ell(\lambda)=m\}$.
For a partition $\lambda=(\lambda_1,\dots,\lambda_m)\in\mathrm{Part}(N,m)$, let
\begin{align*}
A(\lambda)=\{\lambda'\in\mathrm{Part}(N,m+1)|\lambda'=(\lambda_1,\dots,\lambda_{k-1},\lambda_k^1,\lambda_{k}^2,\lambda_{k+1},\dots,\lambda_m)\\\text{ for some }k=1,\dots,m\text{ and }\lambda_k^1+\lambda_{k}^2=\lambda_k\}.
\end{align*}

\begin{dfn}
Let $C(N)$ be the $\mathbb{F}$-vector span of $\mathrm{Part}(N)$ and $C(N)_m$ be the subspace generated by $\mathrm{Part}(N,m)$ for $m=1,\dots,N$.
Let $C(N)_{N+1}=\bm{0}$.
For $m=1,\dots,N-1$, the $\mathbb{F}$-linear map $\partial\colon C(N)_m\to C(N)_{m+1}$ is defined by
\begin{equation*}
\partial(\lambda)=\sum_{\lambda'\in A(\lambda)}\lambda',
\end{equation*}
and $\partial\colon C(N)_N\to C(N)_{N+1}$ is defined to be trivial.
\end{dfn}
\begin{prop}
\label{prop:C(N)-acyclic}
For each $N\in\N$, $(C(N),\partial)$ is a graded chain complex and acyclic if $N>1$.
\end{prop}
\begin{proof}
Let $\lambda'$ be one of the terms of $\partial\circ\partial(\lambda)$.
Then $\lambda'$ has the form 
\begin{equation*}
\lambda'=(\lambda_1,\dots,\lambda_i^1,\lambda_{i}^2,\dots,\lambda_j^1,\lambda_{j}^2,\dots,\lambda_m),
\end{equation*}
with $\lambda_i^1+\lambda_{i}^2=\lambda_i$ and $\lambda_j^1+\lambda_{j}^2=\lambda_j$ for some $i\neq j$ or 
\begin{equation*}
\lambda'=(\lambda_1,\dots,\lambda_i^1,\lambda_{i}^2, \lambda_i^3,\dots,\lambda_m),
\end{equation*}
with $\lambda_i^1+\lambda_{i}^2+\lambda_i^3=\lambda_i$ for some $i$.
It is straightforward to see that each partition $\lambda'$ appears exactly two times.
Therefore, we have $\partial\circ\partial=0$.

When $N=2$, the complex $(C(N),\partial)$ is acyclic since $C(N)$ is just
\begin{equation*}
\mathbb{F}\{(2)\}\to\mathbb{F}\{(1,1)\}\to \bm{0}.
\end{equation*}
For the inductive step, when $N>2$, let $C(N)'$ be the subspace of $C(N)$ generated by the partitions with the form $(1,\lambda_2,\dots,\lambda_m)$ for some $m$.
Then $C(N)'$ is obviously a subcomplex of $C(N)$ and isomorphic to $C(N-1)$ by the corresponding $(1,\lambda_2,\dots,\lambda_m)\mapsto(\lambda_2,\dots,\lambda_m)$.
The quotient complex $C(N)/C(N)'$ is also isomorphic to $C(N-1)$ by the corresponding $(\lambda_1,\lambda_2,\dots,\lambda_m)\mapsto (\lambda_1-1,\lambda_2,\dots,\lambda_m)$.
Since $C(N)'$ and $C(N)/C(N)'$ are acyclic, so is $C(N)$.
\end{proof}

\subsection{The complex of a planar grid}

Let $\mathbb{E}$ be an $n\times n$ \textit{planar} grid of squares such that there are $n$ $O$-markings from top left to bottom right and some $X$-markings. 
Define the horizontal and vertical line $\{\alpha_i\}_{i=1}^{n+1}$, $\{\beta_i\}_{i=1}^{n+1}$, a state $(\alpha_{\sigma(1)}\cap\beta_1,\dots,\alpha_{\sigma(n+1)}\cap\beta_{n+1})$, a domain, and a rectangle in the same way as usual.
In this subsection, we only consider the domains such that every coefficient is at most one.
We remark that, since $\mathbb{E}$ is planar, a rectangle is a connected domain in the plane.
Let $\mathbf{x}_0=\mathbf{x}_0(\mathbb{E})$ be the state $(\alpha_{n+1}\cap\beta_1,\alpha_{n}\cap\beta_2,\dots,\alpha_{1}\cap\beta_{n+1})$, i.e., the collection of the points on the diagonal from top left to bottom right.
Let $D_\mathbb{E}$ be the positive domain whose coefficients of all squares equal one.
We denote by $\pi^+(\mathbb{E})$ the collection of domains $D\in\pi^+(\mathbf{x},\mathbf{x}_0)$ such that $D_\mathbb{E}-D$ is positive and $|(D_\mathbb{E}-D)\cap\mathbb{O}|=|(D_\mathbb{E}-D)\cap\mathbb{X}|=0$ for some state $\mathbf{x}$.
The set $\pi^+(\mathbb{E})$ is a poset with respect to inclusion.

\begin{dfn}
\label{dfn:C(D)}
Let $C(\mathbb{E})$ be the span of $\pi^+(\mathbb{E})$ and define the linear map $\partial_\mathbb{E}\colon C(\mathbb{E})\to C(\mathbb{E})$ by
\begin{equation*}
\partial_\mathbb{E}(D)=\sum_{\substack{D'\in\pi^+(D_\mathbb{E})\\D-D'\text{ is a rectangle}}}D'.
\end{equation*}
Then $C(\mathbb{E})$ is a graded $\mathbb{F}$-vector space with $gr(D)=|\mathbf{x}\cap\mathbf{x}_0|$ for $D\in\pi^+(\mathbf{x},\mathbf{x}_0)$.
\end{dfn}

\begin{prop}
The vector space $C(\mathbb{E})$ is a graded chain complex, i.e. $C(\mathbb{E})=\bigoplus_{i=0}^{n-1}C_i(\mathbb{E})$ and $\partial(C_i(\mathbb{E}))\subset C_{i-1}(\mathbb{E})$.
\end{prop}
\begin{proof}
The differential $\partial_\mathbb{E}$ decreases the grading by one by definition.

Let $D'$ be one of the terms of $\partial_\mathbb{E}\circ\partial_\mathbb{E}(D)$.
Then the positive domain $D-D'$ is the composition of two rectangles, and moreover, it is a hexagon or two disjoint rectangles.
As in the proof that grid complexes are chain complexes (for example, \cite[Lemma 4.6.7]{grid-book}), there are exactly two ways to decompose $D-D'$ into two rectangles.
Thus we have $\partial_\mathbb{E}\circ\partial_\mathbb{E}=0$.
\end{proof}

For a domain $D\in\pi^+(\mathbf{x},\mathbf{x}_0)$, a \textit{northeast corner point} of $D$ is a corner point $x$ of $D$ if the horizontal boundary segment meeting $x$ goes to the left and the vertical boundary segment meeting $x$ goes down.
We define a \textit{southwest corner point} similarly.

Figure \ref{fig: c(E)=C(3)} shows two examples of $C(\mathbb{E})$ that are isomorphic to the complex of partitions $C(3)$.
The red lines represent the domains.
In this case, the southeast corners of the domains correspond to the partitions.

\begin{figure}
    \centering
    \includegraphics[width=1\linewidth]{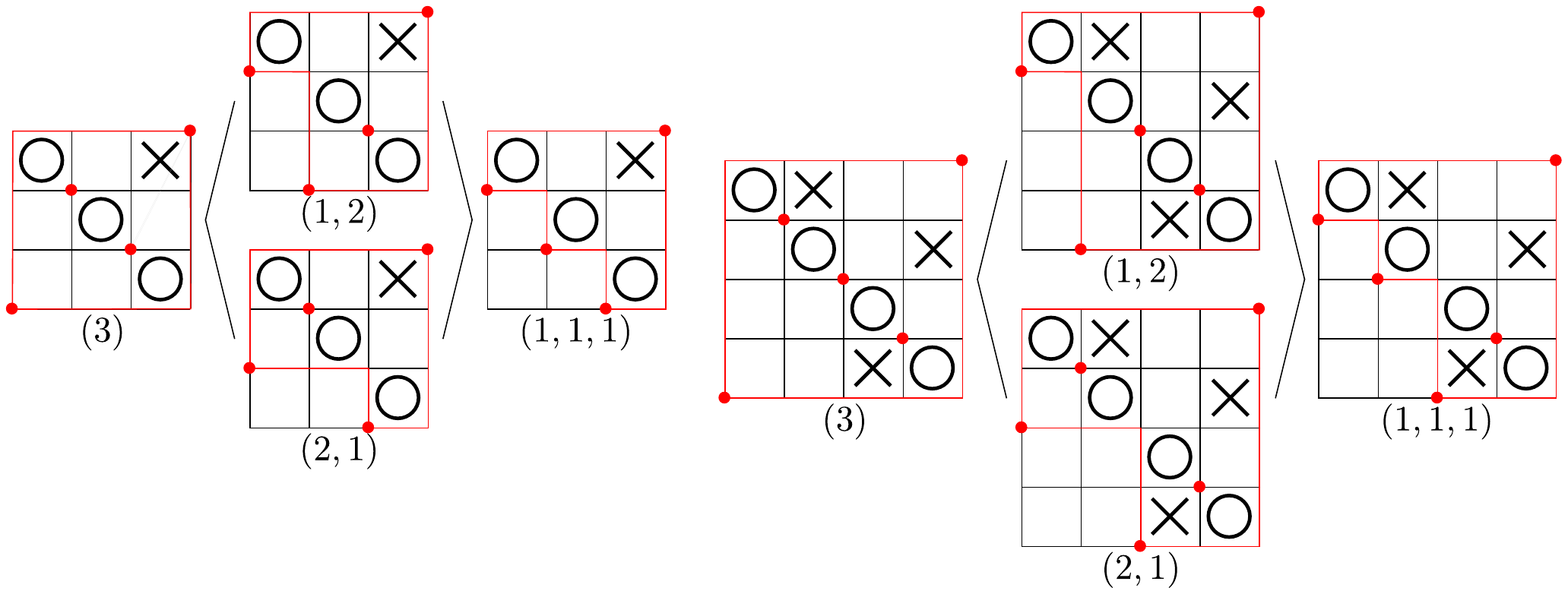}
    \caption{Two examples of $C(\mathbb{E})$ isomorphic to $C(3)$.}
    \label{fig: c(E)=C(3)}
\end{figure}

\begin{prop}
\label{prop:C(E) is acyclic}
\begin{enumerate}[(1)]
    \item If $\pi^+(\mathbb{E})=\{D_\mathbb{E}\}$, then $H(C(\mathbb{E}))$ is one-dimensional.
    \item If $\pi^+(\mathbb{E})\neq\{D_\mathbb{E}\}$ and $\mathbb{E}$ has no pair of $X$-markings symmetric with respect to the diagonal line from top left to bottom right, then $C(\mathbb{E})$ is acyclic.
\end{enumerate}
\end{prop}
\begin{proof}
The case $\pi^+(\mathbb{E})=\{D_\mathbb{E}\}$ is obvious.

Suppose that $\pi^+(\mathbb{E})\neq\{D_\mathbb{E}\}$.
For a domain $D\in\pi^+(\mathbb{E})$ with $D\in\pi^+(\mathbf{x},\mathbf{x}_0)$, let
\begin{equation*}
sw(D)=\{x_i\in \mathbf{x}|x_i \text{ is a southwest corner of } D\text{ and }x_i\notin\mathbf{x}_0 \},
\end{equation*}
and 
\begin{equation*}
ne(D)=\{x_i\in \mathbf{x}|x_i \text{ is a northeast corner of } D\text{ and }x_i\notin\mathbf{x}_0 \}.
\end{equation*}
Then we have $sw(D)+ne(D)+gr(D)=n+1$.
Choose a domain $D_1$ that minimizes the grading.
Since the differential must increase either $sw$ or $ne$ by one, there is another domain $D_2$ such that $\partial(D_2)=D_1$.
Then the two-dimensional vector space $\mathbb{F}\{D_1,D_2\}$ is an acyclic subcomplex of $C(\mathbb{E})$.
Replace $C(\mathbb{E})$ with the quotient $C(\mathbb{E})/\mathbb{F}\{D_1,D_2\}$.

Now, it is sufficient to show that $|\pi^+(\mathbb{E})|\in2\N$.
Since $\mathbb{E}$ has no pair of $X$-markings symmetric with respect to the diagonal, there is a point $x_i\in\mathbf{x}_0$ such that exactly one of $\mathcal{I}(\{x_i\},\mathbb{X})$ and $\mathcal{I}(\mathbb{X},\{x_i\})$ is nonzero.
Then, there is a natural bijection between the set of domains one of whose corners is $x_i$ and the set of other domains.
Thus, we have $|\pi^+(\mathbb{E})|\in2\N$.
\end{proof}

\begin{prop}
\label{prop:H(C(E)) is 1-dim}
If $\mathbb{E}$ has no $X$-marking, we have
\begin{equation*}
H(C_*(\mathbb{E}))\cong 
\begin{cases}
\mathbb{F} & *=0\\
0 & *\neq 0.
\end{cases}
\end{equation*}
\end{prop}
\begin{proof}
Add $n-1$ ``virtual" $X$-markings just above the $O$-markings.
Then $C(\mathbb{E})$ inherits a filtration
\begin{equation*}
C(\mathbf{E})=\mathcal{F}_{n-1}C\supset \mathcal{F}_{n-2}C\supset\dots\supset\mathcal{F}_{0}C\supset\mathcal{F}_{-1}C=0,
\end{equation*}
where the filtration level of the domain $D$ of $C(\mathbb{E})$ is $k$ if $D$ contains $k$ virtual $X$-marking.
The differential preserves this filtration as the filtration level decreases by $m$ when the differential counts domains containing $m$ virtual $X$-markings.
We can observe that the lowest filtration level $\mathcal{F}_0C$ is a one-dimensional vector space and that $\mathcal{F}_kC/\mathcal{F}_{k-1}C$ is isomorphic to the direct sum of some copies of $C(k+1)$.
Then Proposition \ref{prop:C(N)-acyclic} completes the proof.
\end{proof}

\section{The top-2 term of grid homology}
\label{sec:top-2}
In this section, we consider the grid homology at $A=g(K)-2$ for a prime diagonal knot $K$.
Again, we assume \eqref{assumption:A} and \eqref{assumption:B} in Section \ref{sec:top, top-1}.

By Lemma \ref{lem:M0} and \eqref{eq:Alexander-difference}, for any state $\mathbf{x}$ with $(M(\mathbf{x}),A(\mathbf{x}))=(-1,g(K)-2)$, there are exactly two associated square domains $D, D'\in\mathrm{Rect}(\mathbf{x},\mathbf{x}_0)$ such that
\begin{align*}
|D\cap\mathbb{O}|-|D\cap\mathbb{X}|=|D'\cap\mathbb{O}|-|D'\cap\mathbb{X}|=2.
\end{align*}
This equation implies that such a state $\mathbf{x}$ determines a representation of $K$ as the composition of two $2$-tangles, as follows.
\begin{dfn}
For a square domain $D$ with $|D\cap\mathbb{O}|-|D\cap\mathbb{X}|=2$, define a $2$-tangle $T_D$ obtained as follows:
\begin{itemize}
    \item If a row or column of $D$ has both an $O$-marking and an $X$-marking, then connect them by an oriented segment.
    The orientation is taken as in the case of grid diagrams.
    \item If a row has an $O$-marking but no $X$-marking, then draw a ray from the $O$-marking to the right.
    \item If a column has an $O$-marking but no $X$-marking, then draw a ray from the $O$-marking downward.
    \item Finally, for each intersection, assume that the vertical segment (or ray) passes over the horizontal one.
\end{itemize}
\end{dfn}

\begin{lem}
\label{lem:A-2 hexagon in square}
Let $\G$ be a minimal grid diagram representing a prime diagonal knot.
Let $s<-1$.
For any state $\mathbf{x}\in\mathbf{S}_{s,g(K)-2}(\G)$, the associated domain $D\in\pi^+(\mathbf{x},\mathbf{x}_0)$ has coefficients at most one in each square.
If $D$ is connected, then there are a state $\mathbf{y}\in\mathbf{x}_{-1,g(K)-2}(\G)$ and its associated square domain $D'\in\mathrm{Rect}(\mathbf{y},\mathbf{x}_0)$ such that $D$ is contained in $D'$ and 
\begin{equation*}
|D\cap\mathbb{O}|=|D'\cap\mathbb{O}|, |D\cap\mathbb{X}|=|D'\cap\mathbb{X}|.
\end{equation*}
\end{lem}
\begin{proof}
First, we assume that the associated domain $D\in\pi^+(\mathbf{x},\mathbf{x}_0)$ has coefficients at most one in each square.
Let $D'$ be the minimal square domain that contains $D$.
It is sufficient to show that $|(D-D')\cap\mathbb{X}|=0$.
Since $D'$ is a square, we have $|D'\cap\mathbb{O}|\geq|D'\cap\mathbb{X}|$ and hence $|(D'-D)\cap\mathbb{X}|\leq2$.
The primeness of $K$ implies $|(D-D')\cap\mathbb{X}|\neq1$.
If $|(D'-D)\cap\mathbb{X}|=2$, then we have $|D'\cap\mathbb{O}|=|D'\cap\mathbb{X}|$, which implies that $D'$ covers the entire $\G$ or that $\G$ represents a split link.

If $D$ has coefficients more than one in some squares, let $D''$ be the domain obtained from $D$ by setting all non-zero coefficients to one.
Then we have $|D''\cap\mathbb{O}|-|D''\cap\mathbb{X}|\leq|D\cap\mathbb{O}|-|D\cap\mathbb{X}|=2$.
By the same argument above, the minimal square domain $D'$ containing $D''$ represents a nontrivial knot or split link component.
\end{proof}

\begin{dfn}
For $\mathbf{x}\in\mathbf{S}_{-1,g(K)-2}(\G)$, the associated square domain $D\in\mathrm{Rect}(\mathbf{x},\mathbf{x}_0)$ is called \textit{essential} if there is no empty rectangle $R\in\mathrm{Rect}^\circ(\mathbf{x},\mathbf{x}')$ contained in $D$ for any $\mathbf{x}'$ satisfying $|R\cap\mathbb{O}|=|R\cap\mathbb{X}|=0$.
\end{dfn}

We remark that any essential square domain of $\mathbf{x}\in\mathbf{S}_{-1,g(K)-2}(\G)$ is larger than a $4\times4$ square if $\G$ represents a prime diagonal knot.
The set of essential domains is a poset with respect to inclusion.

\begin{lem}
\label{lem:non-essential}
Let $\G$ be a minimal diagonal grid diagram representing a prime diagonal knot.
If a square domain $D$ with $|D\cap\mathbb{O}|-|D\cap\mathbb{X}|=2$ is larger than a $2\times2$ square and not essential, then the associated tangle $T_D$ decomposes into two tangles.
\end{lem}
\begin{proof}
Let $D\in\pi^+(\mathbf{x},\mathbf{x}_0)$.
By definition, there is an empty rectangle $R\in\mathrm{Rect}^\circ(\mathbf{x},\mathbf{x}')$.
Then $R$ naturally determines a decomposition of $D$ into two squares $D_1, D_2$ and two rectangles.
We denote the other rectangle by $R'$.

If the width or height of $R$ is one, then one of $D_1$ and $D_2$, say $D_2$, is a $1\times1$ square.
The minimality of $\G$ implies that the tangle $T_D$ is obtained from $T_{D_2}$ by adding one crossing.

Suppose that $R$ has both width and height greater than one.
If $|R'\cap\mathbb{X}|=1$, then we have $|D_1\cap\mathbb{O}|-|D_1\cap\mathbb{X}|=1$ or $|D_2\cap\mathbb{O}|-|D_2\cap\mathbb{X}|=1$, which contradicts the primeness of $K$.
Similarly, we can rule out $|R'\cap\mathbb{X}|=0$ because $\G$ represents a prime knot.
Therefore, we have $|R'\cap\mathbb{X}|=2$ and hence each of $D_1$ and $D_2$ represents a $2$-tangle.
\end{proof}

\begin{lem}
\label{lem:non-essential - integer tangle}
Let $\G$ be a minimal grid diagram representing a prime diagonal knot.
If a non-essential square domain $D$ with $|D\cap\mathbb{O}|-|D\cap\mathbb{X}|=2$ does not contain any essential square domain, then the associated tangle $T_D$ represents an integer tangle.
\end{lem}
\begin{proof}
We remark that since $\G$ is minimal, no $X$-marking is adjacent to any $O$-marking.

We proceed by induction on the size of $D$, denoted by $l$.
The case $l=2,3$ is straightforward.

Assume that the claim is true for any domain smaller than an $l\times l$ square.
Take a decomposition of $D$ into two squares $D_1, D_2$ and two rectangles $R,R'$ as in the proof of Lemma \ref{lem:non-essential}.
If one of $D_1$ and $D_2$ is a $1\times1$ square, then the larger domain represents an integer tangle, and so does $D$.
Suppose that $R$ has both width and height greater than one.
Now $D_1, D_2$ are not essential.
By the induction hypothesis, the associated tangles $T_{D_1}$ and $T_{D_2}$ are integer tangles.
Since $T_D$ is a composition of $T_{D_1}$ and $T_{D_2}$, $T_D$ is also an integer tangle.
\end{proof}

\begin{dfn}
For $l\geq2$, let $\mathbb{E}_l$ be the $l\times l$ planar grid of squares satisfying the followings:
\begin{itemize}
    \item There are $l$ $O$-markings from top left to bottom right,
    \item there are $(l-2)$ $X$-markings such that each of them is two squares to the right of an $O$-marking.
\end{itemize}
Similarly, let $\mathbb{E}'_l$ be the $l\times l$ planar grid of squares satisfying the followings:
\begin{itemize}
    \item There are $l$ $O$-markings from top left to bottom right,
    \item there are $(l-2)$ $X$-markings such that each of them is two squares to the left of an $O$-marking.
\end{itemize}
The square domains $\mathbb{E}_l$ and $\mathbb{E}'_l$ are shown in Figure \ref{fig:El}.
\end{dfn}

\begin{figure}[htbp]
\centering
\includegraphics[scale=0.5]{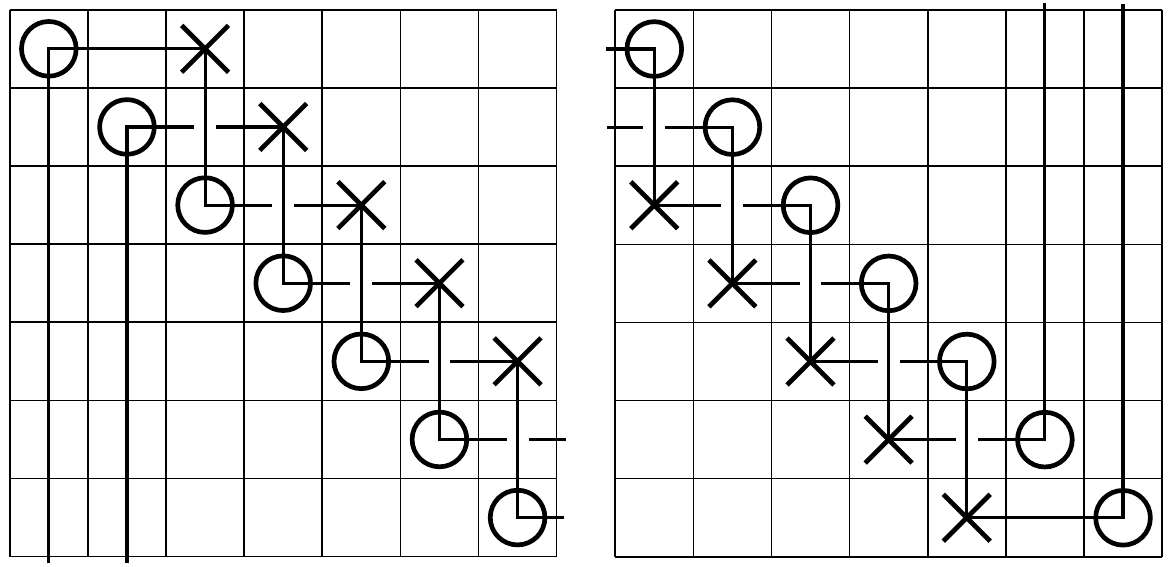}
\caption{The square domains $\mathbb{E}_l$ (left) and $\mathbb{E}'_l$ (right).}
\label{fig:El}
\end{figure}

\begin{dfn}
A grid diagram $\G$ is called \textit{standard} if any maximal non-essential square domain $D$ with $|D\cap\mathbb{O}|-|D\cap\mathbb{X}|=2$ is naturally identified with $\mathbb{E}_l$ for some $l$.
\end{dfn}

\begin{lem}
Any prime diagonal knot admits a standard diagonal grid diagram.
Furthermore, there is a minimal standard diagonal grid diagram among all diagonal grid diagrams representing the knot.
\end{lem}
\begin{proof}
Let $\G$ be a minimal diagonal grid diagram.
By Lemma \ref{lem:non-essential - integer tangle}, each maximal non-essential square domain $D$ represents an integer tangle.
Since $\G$ is minimal, $D$ must be $\mathbb{E}_l$ or $\mathbb{E}'_l$.
If the square domain is $\mathbb{E}'_l$, then there is a sequence of grid moves that transforms it into $\mathbb{E}_l$ (see Figure~\ref{fig:El2}).
The resulting grid diagram is the same size as $\G$ since $\mathbb{E}_l$ is minimal among diagonal planar grids representing an integer tangle.
\end{proof}

\begin{figure}[htbp]
\centering
\includegraphics[scale=0.5]{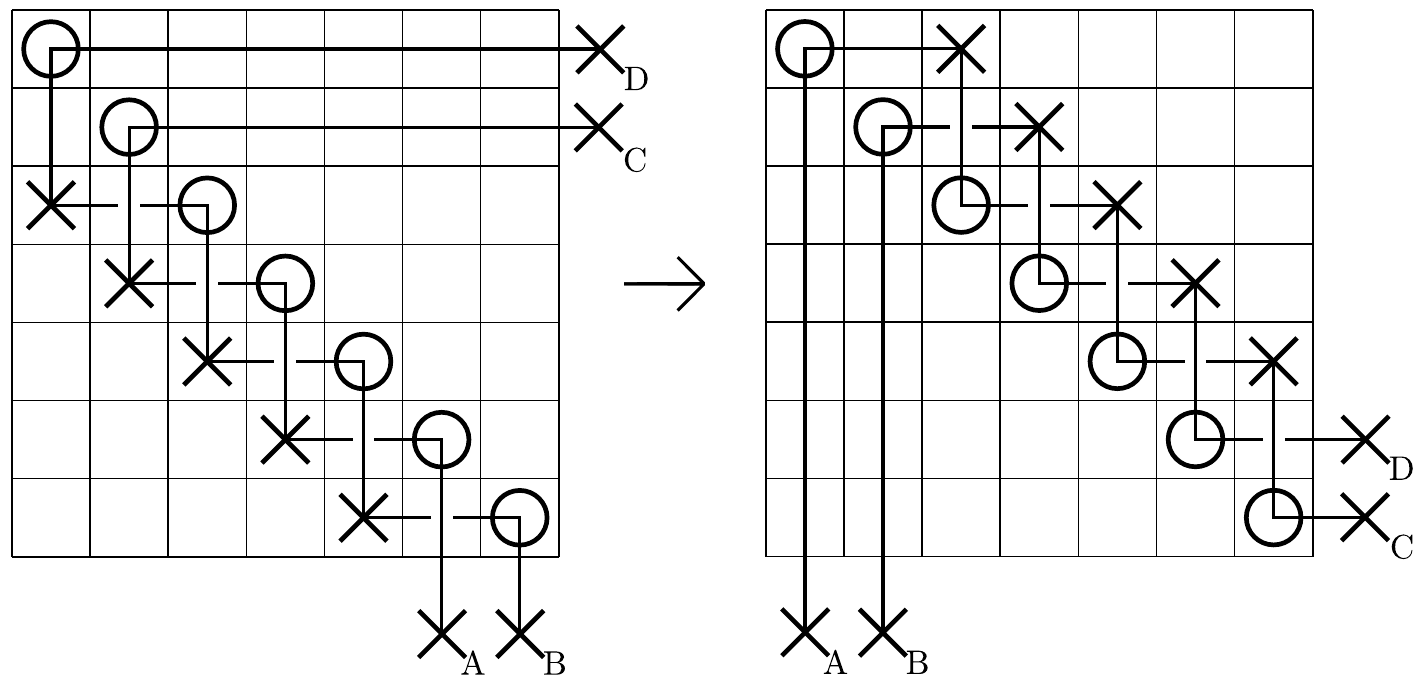}
\caption{$\mathbb{E}'_l$ can be transformed into $\mathbb{E}_l$.}
\label{fig:El2}
\end{figure}

\begin{lem}
\label{lem:essential in non-essential}
Let $\G$ be a minimal, standard, diagonal grid diagram representing a prime diagonal knot $K$.
If an essential square domain $D_1$ and a non-essential domain $D_2$ with
\begin{equation*}
|D_1\cap\mathbb{O}|-|D_1\cap\mathbb{X}|=|D_2\cap\mathbb{O}|-|D_2\cap\mathbb{X}|=2
\end{equation*}
intersect, then one is contained in the other.
\end{lem}
\begin{proof}
Let $D_0$ be the square domain that is the intersection of $D_1$ and $D_2$.
Suppose that $D_1-D_0$ and $D_2-D_0$ are nontrivial.
Since $D_2$ is naturally identified with $\mathbb{E}_l$ and $D_0$ is contained in $D_2$, $D_0$ must be identified with $\mathbb{E}_{l'}$ for some $l'<l$.
Suppose that $D_0$ is in the lower right of $D_1$, and in the upper left of $D_2$.
Then the bottom row of $D_1$ does not contain an $X$-marking, which implies $D_1$ is not essential.
If $D_0$ is in the lower right of $D_2$ and in the upper left of $D_1$, then a similar argument implies that $D_1$ is not essential.
\end{proof}

\begin{prop}
\label{prop:A-2, s not -1}
Let $\G$ be a minimal, standard, diagonal grid diagram representing a prime diagonal knot $K$.
If $\G$ has at least one essential domain $D$ with $|D\cap\mathbb{O}|-|D\cap\mathbb{X}|=2$, then we have
\begin{equation*}
\widetilde{GH}_*(\G,g(K)-2)\cong
\begin{cases}
\mathbb{F}^m & *=-1\\
\mathbb{F}^{\binom{n}{2}} & *=-2\\
0 & *\neq -1,-2,
\end{cases}
\end{equation*}
where $n$ is the size of $\G$ and 
\begin{equation*}
m=\#\{\mathbf{x}\in\mathbf{S}_{-1,g(K)-2}\mid D_1,D_2\in\pi^+(\mathbf{x},\mathbf{x}_0)\text{ are both essential}\}.
\end{equation*}
\end{prop}
\begin{proof}

We will show that, by successively taking quotients of $\widetilde{GC}_*(\G,g(K)-2)$ by acyclic subcomplexes, the complex can be reduced to a simple chain complex such that its generating state $\mathbf{x}$ satisfies one of the following conditions:
\begin{enumerate}[(1)]
    \item $\pi^+(\mathbf{x},\mathbf{x}_0)$ consists of two essential square domains (Figure \ref{fig:13n241_2}),
    \item One of the associated domains of $\mathbf{x}$ is connected and contains exactly two $O$-markings (on the left of Figure \ref{fig:2211}), or
    \item One of the associated domains of $\mathbf{x}$ consists of two disjoint $1\times1$ squares (on the right of Figure \ref{fig:2211}).
\end{enumerate}
Take any state $\mathbf{x}$ in $\widetilde{GC}_*(\G,g(K)-2)$ satisfying none of them.
We remark that every associated domain of $\mathbf{x}$ has coefficients at most one on each square.
Then using Lemma \ref{lem:A-2 hexagon in square}, we can take an associated domain $D$ of $\mathbf{x}$ and a square domain $D_\mathbf{x}$ such that
\begin{itemize}
    \item $D_\mathbf{x}$ contains $D$.
    \item $D_\mathbf{x}$ is not essential.
    \item $|D_\mathbf{x}\cap\mathbb{O}|=|D\cap\mathbb{O}|$.
    \item $|D_\mathbf{x}\cap\mathbb{X}|=|D\cap\mathbb{X}|$.
\end{itemize}
Since $D_\mathbf{x}$ is not essential, Proposition \ref{prop:C(E) is acyclic} gives an acyclic chain complex $C(D_\mathbf{x})$ (Definition \ref{dfn:C(D)}).
In general, this complex is not necessarily a subcomplex of $\widetilde{GC}_*(\G,g(K)-2)$.

For an essential square domain $D$ with $|D\cap\mathbb{O}|-|D\cap\mathbb{X}|=2$, let $\mathbf{S}_D$ be the collection of the states not satisfying the above three conditions, one of whose associated domains is contained in $D$ and contains all essential domains properly contained in $D$.
Since $D$ is essential, the span of $\mathbf{S}_D$, denoted by $C_D$, is a subcomplex of $\widetilde{GC}_*(\G,g(K)-2)$.
We remark that for any state $\mathbf{x}\in\mathbf{S}_D$, there is a unique associated domain contained in $D$.
If $\mathbf{S}_D\neq\emptyset$, then there is a unique maximal non-essential square domain $D_0$ associated with some state of $\mathbf{S}_D$.
Since $D$ is essential and $D_0$ is maximal, the acyclic complex $C(D_0)$ is a subcomplex of $C_D$.
After taking the quotient complex $C_D/C(D_0)$, the second largest non-essential square domain $D_1$ (note that this is not necessarily unique) gives an acyclic complex $C(D_1)$ which is a subcomplex of the quotient.
By repeating this procedure, we conclude that $C_D$ is acyclic.

Now we obtain the acyclic subcomplex of $\widetilde{GC}_*(\G,g(K)-2)$ defined by
\begin{equation*}
C=\bigoplus_{D\colon essential}C_D.
\end{equation*}
Let $C'=\widetilde{GC}_*(\G,g(K)-2)/C$ be the quotient complex.

Suppose that there is a state $\mathbf{S}_{s,g(K)-2}$ for some $s<-1$ that is not a generator of $C$.
Let $\mathbf{S}'$ be the collection of such states not satisfying the above three conditions.
We remark that by Lemma \ref{lem:essential in non-essential}, for any state in $\mathbf{S}'$, there is a unique associated domain $D$ that is disjoint from all essential domains, and there is a non-essential square domain containing $D$.
Then the span of $\mathbf{S}'$, denoted by $C_{\mathbf{S}'}$, is a subcomplex of $C'$.
By the same argument as $C_D$, we see that $C_{\mathbf{S}'}$ is acyclic.
Let $C''=C'/C_{\mathbf{S}'}$ be the quotient complex.

Now each generating state $\mathbf{x}$ of $C''$ satisfies one of:
\begin{enumerate}[(1)]
    \item $\pi^+(\mathbf{x},\mathbf{x}_0)$ consists of two essential square domains,
    \item One of the associated domains of $\mathbf{x}$ is connected and contains exactly two $O$-markings, or
    \item One of the associated domains of $\mathbf{x}$ consists of two disjoint $1\times1$ squares.
\end{enumerate}
For the first case, each state individually represents a nontrivial homology class of $\widetilde{GH}_{-1}(\G,g(K)-2)$ and therefore $\widetilde{GH}_{-1}(\G,g(K)-2)\cong\mathbb{F}^m$.
For the states in the second case, three states, one associated with one square domain and two associated with hexagon domains, form a subcomplex of $C'$, and its homology is one-dimensional by Proposition \ref{prop:H(C(E)) is 1-dim}.
For the third case, each state $\mathbf{x}$ satisfies $\widetilde{\partial}(\mathbf{x})=0$ and represents a nontrivial homology class of $\widetilde{GH}_{-2}(\G,g(K)-2)$.
Gathering the second and third case, we obtain $\widetilde{GH}_{-2}(\G,g(K)-2)\cong\mathbb{F}^{\binom{n}{2}}$.
\begin{figure}[htbp]
\centering
\includegraphics[width=1\linewidth]{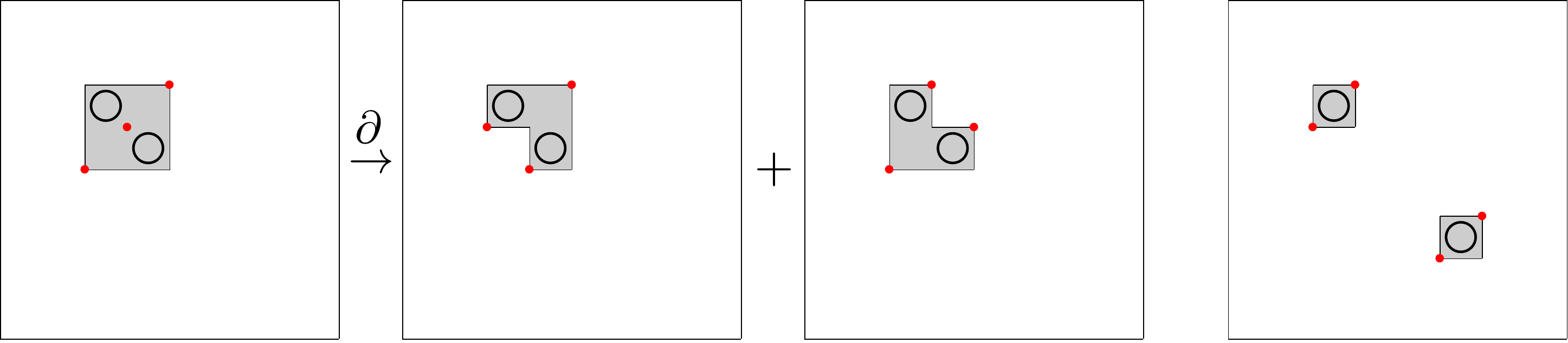}
\caption{Left: the three states surrounding exactly two $O$-markings. Right: the state whose associated domain consists of two $1\times1$ squares.}
\label{fig:2211}
\end{figure}
\end{proof}

\begin{prop}
\label{prop:essential = integer tangle}
Let $D$ be a diagonal square domain satisfying
\begin{itemize}
    \item $|D\cap\mathbb{O}|-|D\cap\mathbb{X}|=2$.
    \item $D$ contains no square domain $D'$ larger than a $1\times1$ square such that $|D'\cap\mathbb{O}|-|D'\cap\mathbb{X}|=1$.
    \item The size of $D$ is minimal among diagonal square domains representing the same tangle.
\end{itemize}
Then, the associated tangle $T_D$ is an integer tangle if and only if $D$ is not essential and contains no essential square domain.
\end{prop}
\begin{proof}
Suppose that $T_D$ is an integer tangle and that $D$ is essential or contains at least one essential square domain.
Let $l$ be the size of $D$.
Construct the $2l\times 2l$ diagonal grid diagram $\G$ such that the upper-left and lower-right $l\times l$ blocks of $\G$ are identified with $D$.
Since $T_D$ is an integer tangle, $\G$ represents a $(2,q)$ torus knot for some $q$.
We remark that each of the upper-right and lower-left $l\times l$ blocks of $\G$ contains exactly two $X$-markings, and their locations can be chosen freely as long as $\G$ represents a knot.
Since $\G$ has essential square domains, Proposition \ref{prop:A-2, s not -1} concludes that $\widehat{GH}_{*}(\G,g(K)-2)$ is trivial for $*\neq-1$.
This contradicts that $\G$ represents a $(2,q)$ torus knot.
Now Lemma \ref{lem:non-essential - integer tangle} completes the proof.
\end{proof}

\begin{proof}[Proof of Theorem \ref{thm:diagonal top-2}]
Combine Propositions \ref{prop:A-2, s not -1} and \ref{prop:essential = integer tangle}, and Theorems \ref{thm:str-GH-of-diagonal-knots} and \ref{thm:hat-tilde}.
\end{proof}

\section{The unknotting numbers of diagonal knots}
\label{sec:unknotting-diagonl}
In this section, we prove $u(K)=g(K)$ for diagonal knots.

\begin{proof}[Proof of Theorem \ref{thm:diagonal-u=g}]
The idea is that if we exchange an appropriate crossing of a diagonal knot $K$, we obtain a diagonal knot $K'$ such that $g(K')=g(K)-1$.

Let $\G$ be a minimal grid diagram representing a diagonal knot $K$.
Consider the diagram of $K$ obtained from $\G$ by drawing horizontal and vertical segments.
Let $c$ be the number of crossings of the diagram.
Let $F$ be the canonical Seifert surface of $K$ obtained by Seifert's algorithm and $d$ be the number of Seifert circles.
By the proof of \cite[Theorem 1]{Some-classes-of-fibered-links}, $F$ is the fiber of the fibration and thus $F$ attains the minimal genus.
Since $\G$ is diagonal, there is one Seifert circle that contains the others.

Choose one of the innermost Seifert circles $d_i$.
Let $m$ be the number of half-twisted bands connecting to the circle $d_i$.
Since $\G$ is minimal, we have $m>1$.
Since $\G$ is diagonal, all bands are attached from the upper right or lower left of the disk $d_i$.
The minimality of $\G$ ensures that there is a crossing of $K$ whose distance from the northeast or southwest corner of $d_i$ is $1/2$.
If we exchange the crossing, we can obtain a diagonal knot with the genus $g(K)-1$, as follows.
\begin{enumerate}[(1)]
    \item If $m>3$, the resulting diagonal knot has $c-2$ crossings and $d$ Seifert circles (top of Figure \ref{fig:Seifert}).
    \item If $m=3$, the resulting diagonal knot has $c-3$ crossings and $d-1$ Seifert circles (middle of Figure \ref{fig:Seifert}).
    \item If $m=2$, the resulting diagonal knot has $c-2$ crossings and $d$ Seifert circles (bottom of Figure \ref{fig:Seifert}).
\end{enumerate}
By repeating this procedure, we conclude that $u(K)=g(K)$.
\end{proof}
\begin{figure}[htbp]
\centering
\includegraphics[scale=0.5]{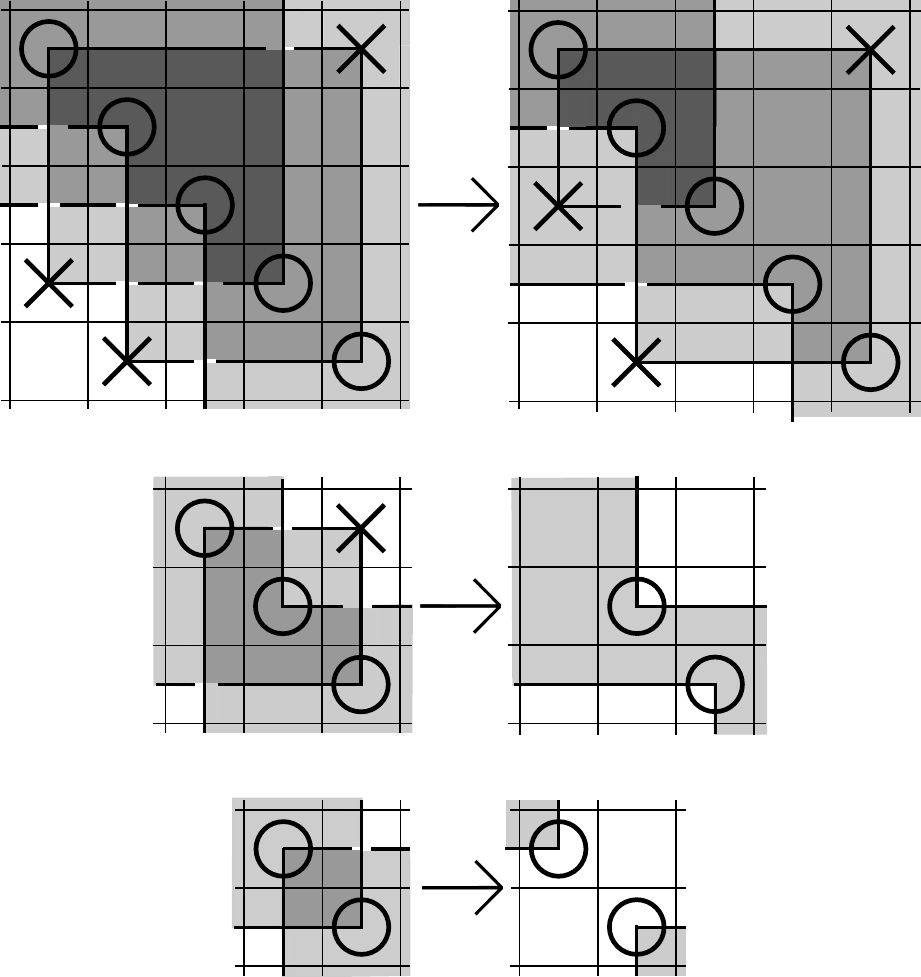}
\caption{Unknotting operations for each case in the proof of Theorem \ref{thm:diagonal-u=g}. The rows correspond to cases (1), (2), and (3), respectively.}
\label{fig:Seifert}
\end{figure}

\appendix
\section{Diagonal grid diagrams for some positive braids}
\label{sec:appendix}

We give diagonal grid diagrams for some positive braid knots.
By rotating a braid $90^\circ$ clockwise, we can easily get a grid diagram representing the closure of a braid.
We remark that if a grid diagram represents an oriented link $\vec{L}$, then its reflection across the diagonal represents $-\vec{L}$, where $-\vec{L}$ is the orientation reversal of every component of $\vec{L}$.
Figure \ref{fig:10_139} shows the procedure to obtain a diagonal grid diagram representing the knot $10_{139}$.
\begin{figure}[htbp]
\centering
\includegraphics[width=1\linewidth]{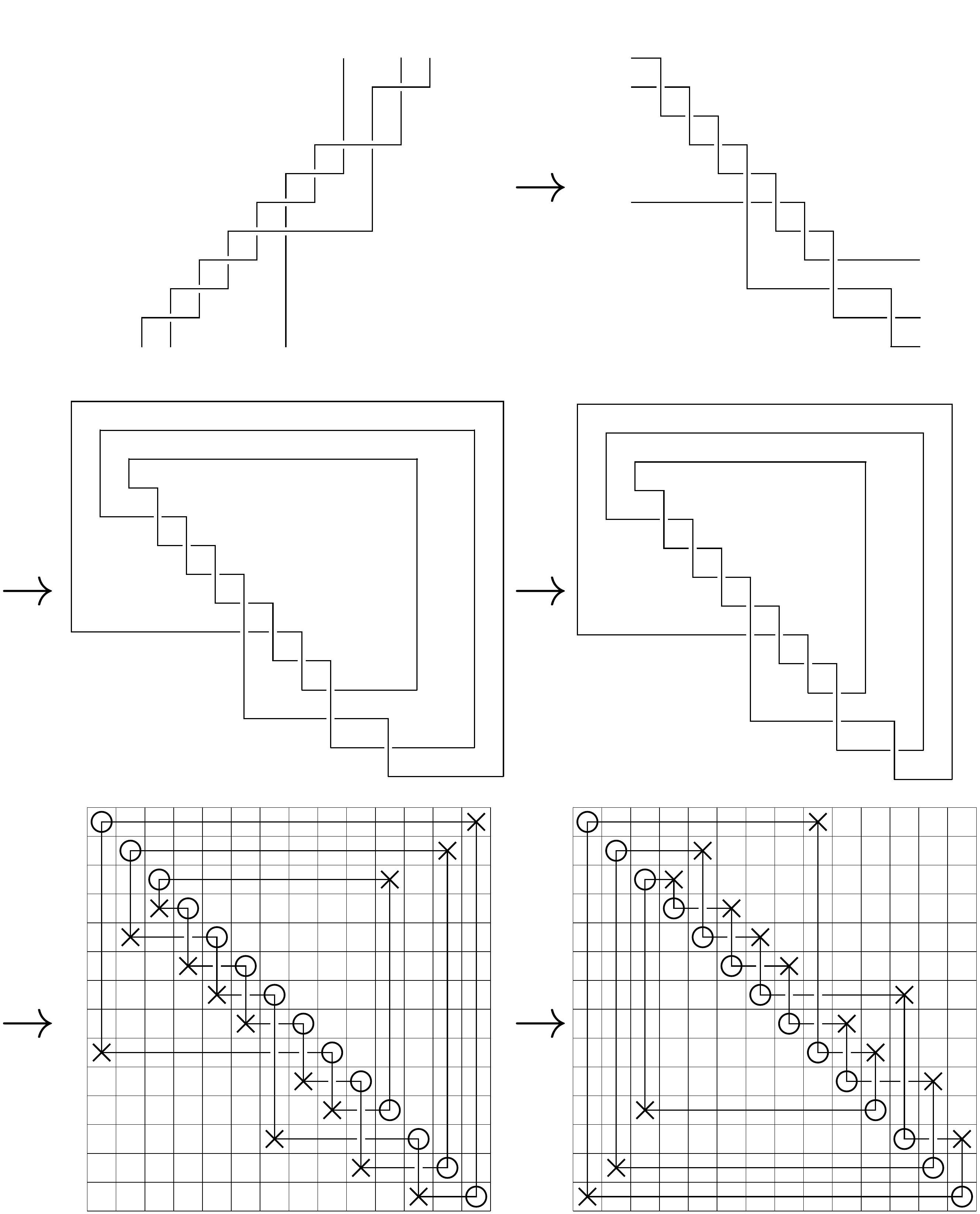}
\caption{A diagonal grid diagram representing the knot $10_{139}$ which is the closure of the braid $\sigma_1^4\sigma_2\sigma_1^3\sigma_2^2$. The grid diagrams in the bottom left and bottom right represent $-10_{139}$ and $10_{139}$ respectively.}
\label{fig:10_139}
\end{figure}

\begin{prop}
Let $[m]=\sigma_1^m\sigma_2\in\mathcal{B}_3$ be the word of the braid group (Figure \ref{fig:braid}).
If a link has a braid representation as $[m_1]\dots[m_i]\sigma_2^l$ for some positive integers $m_1,\dots,m_i$ and non-negative integer $l$, then there exists a diagonal grid diagram representing it.
\end{prop}
\begin{proof}
This immediately follows from the idea of the above procedure (See Figure \ref{fig:braid_diagonal}).
\end{proof}

By this Proposition, the knots
\begin{equation*}
10_{139}, 12n_{i}\ (i=242, 472, 574, 725), 14n_{j}\ (j=6022,12201,15856,18079,21324,24551)
\end{equation*}
that are not torus knots are diagonal.
Prime positive braid knots with the genus up to $6$ are in \cite{positive-braid-knot-list}. See also \cite{Checkerboard-graph-monodromies}.
We remark that the knot $12n_{242}$ is an $L$-space knot and  $12n_{472}$, $12n_{574}$, and $12n_{725}$ are not.

\begin{figure}[htbp]
\centering
\includegraphics[scale=0.5]{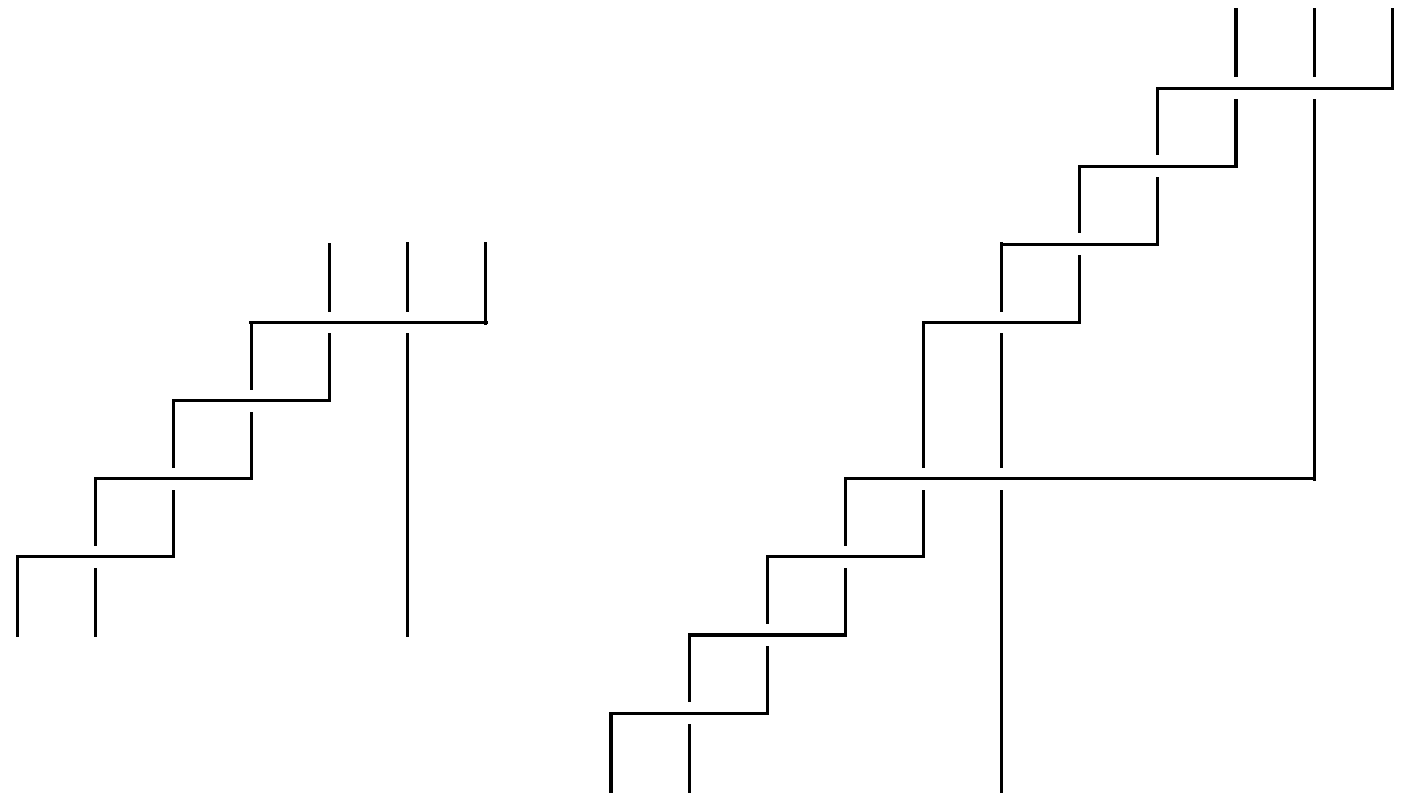}
\caption{Left: the braid $[m]:=\sigma_1^m\sigma_2$. Right: the composite braid $[m][m']$.}
\label{fig:braid}
\end{figure}
\begin{figure}[htbp]
\centering
\includegraphics[width=1\linewidth]{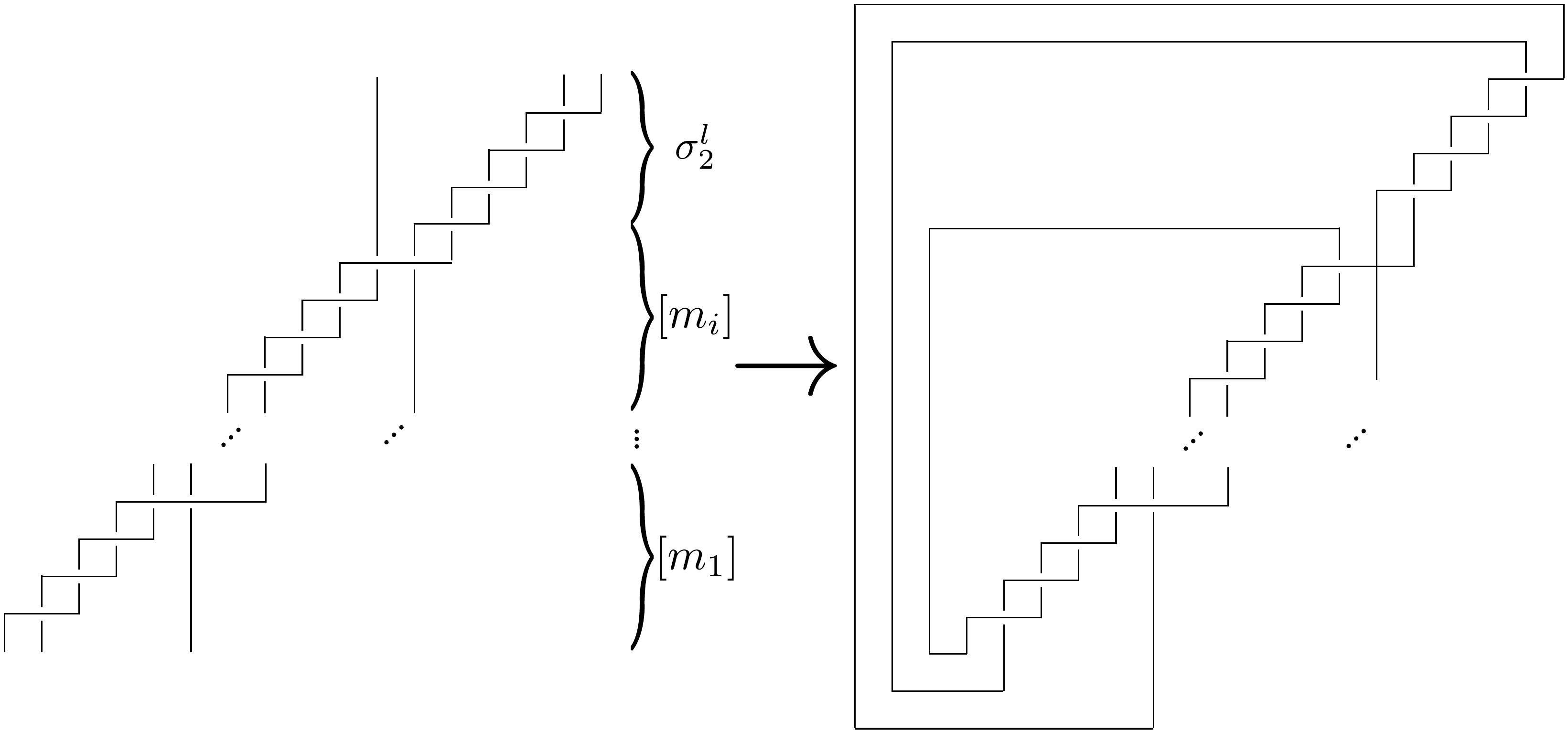}
\caption{The positive braid $[m_1]\dots[m_i]\sigma_2^l$ and its closure $L$.}
\label{fig:braid_diagonal}
\end{figure}

\bibliography{grid}
\bibliographystyle{plain} 

\end{document}